\documentclass{article}
\usepackage{amsmath}
\usepackage{amssymb}
\usepackage{float}
\usepackage{graphicx}
\usepackage{caption}
\usepackage{subcaption}
\usepackage{etoolbox}
\usepackage{amsthm}
\usepackage{ytableau}
\usepackage{tikz}
\usepackage{makecell}
\usepackage[utf8]{inputenc}
\usepackage{multirow}
\usepackage{xfrac}
\usepackage{upgreek}

\usepackage[a4paper,margin=2.5cm,footskip=0.25in]{geometry}

\DeclareMathOperator{\sgn}{sgn}
\DeclareMathOperator{\Ai}{Ai}
\DeclareMathOperator{\Bi}{Bi}

\newtheorem{theorem}{Theorem}[section]
\newtheorem{corollary}[theorem]{Corollary}
\newtheorem{lemma}[theorem]{Lemma}
\newtheorem{proposition}[theorem]{Proposition}
\newtheorem{remark}[theorem]{Remark}
\theoremstyle{definition}

\counterwithin*{equation}{section}
\counterwithin*{equation}{section}
\renewcommand*{\theequation}{\ifnumgreater{\value{section}}{0}{\thesection.}{\thesection.}\arabic{equation}}%

\newcommand{\beq}{\begin{equation}}  
\newcommand{\eeq}{\end{equation}}  
\newcommand\dd{\mathrm{d}}

\newcommand{\Z}{{\mathbb Z}}
\newcommand{\C}{{\mathbb C}}

\newcommand{\rd}{\mathrm{d}}
\newcommand{\rh}{\mathrm{h}}

\newcommand{\rp}{\mathrm{p}}
\newcommand\la{{\lambda}}

\newcommand\ka{{\kappa}}
\newcommand\al{{\alpha}}
\newcommand\be{{\beta}}
\newcommand\ze{{\zeta}}

\newcommand{\si}{\sigma}

\title{Confluent Darboux transformations 
and Wronskians for algebraic solutions of 
the Painlev\'e III ($D_7$) equation} 

\author{J.W.E. Harrow  and A.N.W. Hone\footnote{Corresponding author e-mail: A.N.W.Hone@kent.ac.uk}~\\
School of Mathematics, Statistics \& Actuarial Science~\\
    University of Kent, Canterbury CT2 7NF, UK}

\begin{document}

\maketitle

\begin{abstract} 
We describe the use of confluent Darboux transformations for Schr\"odinger operators, and how they give rise to explicit Wronskian 
formulae for certain algebraic solutions of 
Painlev\'e equations. As a preliminary illustration, we briefly describe how the 
 Yablonskii-Vorob'ev polynomials arise in this way, thus providing well-known expressions for the tau functions of the rational solutions of the Painlev\'e II equation. We then proceed to apply the method to obtain the main result, namely a new  Wronskian representation  for the Ohyama polynomials, which correspond to the algebraic solutions  of the Painlev\'e III equation of type $D_7$. 
\end{abstract}

\section{Introduction}

In its original form as discovered by Darboux 
\cite{DarbouxOriginal}, the Darboux transformation is a relation between the solutions and coefficients of a pair of linear ordinary differential equations of second order. Hence, in the case of a  1D Schr\"odinger equation, it gives a  
technique for obtaining a new  eigenfunction and potential, starting from some initially known eigenfunction and potential. As such, this technique has led to a powerful algebraic approach for constructing families of exactly solvable potentials, within the framework of supersymmetric quantum mechanics \cite{susyqm}. Furthermore, the generalization of the Darboux transformation to the case of higher order linear operators, or compatible matrix linear systems (zero curvature equations), has resulted in an extremely effective tool for 
deriving explicit solutions of integrable 
nonlinear partial differential equations (PDEs), including soliton solutions \cite{matveevsalle}. 

The focus of this article is on explicit solutions of  nonlinear ordinary differential  equations (ODEs), rather than PDEs. Specifically, here we will be concerned with solutions of Painlev\'e equations, which are second order  ODEs of the general form 
\beq\label{painleveeqn}
\frac{\dd^2 q}{\dd z^2} = F\left(z, q, \frac{\dd q}{\dd z}\right) , 
\eeq 
where the function $F$ on the right-hand side is rational in the first derivative of $q$, algebraic in $q$, and analytic in $z$, having the property that all solutions are meromorphic away from a finite number of fixed critical points (which are determined by the equation itself). Up to certain coordinate transformations, the Painlev\'e equations are classified into six canonical forms, referred to as Painlev\'e I-VI, given by particular functions $F$ depending on certain  parameters, which are at most four in number.  (See chapter XIV in \cite{ince} for details of this classification.)

It is known that the general solution of each Painlev\'e equation is a higher transcendental function, which cannot be expressed in terms of simpler functions, e.g. elliptic functions or classical special functions given by solutions of linear ODEs. For this reason, the solutions of Painlev\'e equations should be regarded as quintessentially nonlinear special functions, which provide solutions of many fundamental problems appearing in a diverse areas of application, including probability theory, random matrices, quantum gravity, orthogonal polynomials, and asymptotics of PDEs (see \cite{pac} and references therein). 
However, there are certain parameter values for which Painlev\'e equations admit special solutions that are expressed in terms of simpler functions. For instance, the 
Painlev\'e II equation, which is usually written in the form  
\beq\label{pii}
\frac{\dd^2 q}{\dd z^2} =2q^3 + zq+\al,  
\eeq 
has a set  of particular  solutions $q_n(z)$  
for integer parameter values $\al = n\in\Z$, 
given by the sequence of rational functions  
with $q_{-n}(z)=-q_n(z)$, as 
shown in Table \ref{piirat}; and for half-integer values of $\alpha$, it has a 1-parameter family of special solutions that can be written in terms of Airy functions and their derivatives. 

\begin{table}[]
    \centering 
\begin{tabular}{ |c||c|c|c|c| c| }
 \hline &&&&&
 \\
   $\al$ & \quad 0 \quad & \quad 1 \quad & 2 & 3 & 4 \\
 &&&&&\\\hline
 &&&&&\\
 $q$ & 0 & $-\tfrac{1}{z}$ & 
 $\frac{1}{z}- \frac{3z^2}{z^3+4}$ & 
  $\frac{3z^2}{z^3+4}- \frac{6z^2(z^3+10)}{z^6+20z^3-80}$ 
  & $\frac{6z^2(z^3+10)}{z^6+20z^3-80}- 
  \frac{10(z^9+42z^6+1120)}{z(z^9+60z^6+11200)}$
  \\
  &&&&& \\
 \hline
\end{tabular}
    \caption{Rational solutions of Painlev\'e II.}
    \label{piirat}
\end{table}

The standard way to obtain sequences of solutions such as these 
is via 
the application of B\"acklund transformations (BTs), which are discrete symmetries of 
Painlev\'e equations that map solutions to solutions, while changing the parameters. In the case of Painlev\'e II, there are two independent symmetries of this kind, given by 
\beq\label{piibts} \begin{array}{rll}
S: & q\mapsto - q,  \qquad  \qquad\qquad & \al \mapsto -\al \\
T: & q\mapsto \tilde{q}=q+ \big(\al+\tfrac{1}{2}\big)
/\big(\frac{\dd q }{\dd z}+q^2+\tfrac{z}{2}\big), &
 \al\mapsto -\al-1.   
\end{array}
\eeq
The composition $S\circ T$ sends 
$q\mapsto -\tilde{q}$, maps the 
parameter $\al\mapsto \al +1$, and generates the 
sequence  in Table \ref{piirat}, starting from the seed solution $y=0$ when $\al=0$. It has been known since the work of Okamoto 
that the BTs of each   
Painlev\'e equation are associated with birational actions of an (extended) affine Weyl group on the space of initial conditions, 
with the parameters corresponding to root variables (see \cite{okamoto}, for instance).  
It was subsequently found by Sakai 
that the space of initial conditions for 
both continuous and discrete Painlev\'e equations can be identified with a 
smooth rational surface ${\cal X}$, whose anti-canonical class determines a pair of orthogonal  affine root subsystems inside the $E_8^{(1)}$ root lattice: one of these root subsystems corresponds to %
the exceptional divisors obtained via a blowing up procedure, and determines the surface type; while the other is associated with the affine Weyl group that determines the discrete symmetries of the equation (and hence its BTs).    
For the Painlev\'e II equation (\ref{pii}), the surface $\cal X$ is of  type  $E_7^{(1)}$, while the symmetry type is $A_1^{(1)}$, so that in particular there is only one parameter $\al$,  corresponding to  the fact that the $A_1$ root system has rank 1. 

Another essential aspect of Okamoto's work
was the representation of Painlev\'e equations 
as non-autonomous Hamiltonian systems, and the tau function associated with this representation. In the case of Painlev\'e II, we have a variable $p$ which is the canonically conjugate momentum associated with $q$, and the pair of Hamilton's equations 
\beq\label{hampii}
\begin{array}{rclcc}
    q' & = &-q^2-p-\tfrac{z}{2} 
   & =&\frac{\partial h }{\partial p}   \\
    p' & = & 2pq-\ell & = & - \frac{\partial h }{\partial q}, 
\end{array}
\eeq
where the prime denotes the time evolution ($z$ derivative), while  
the Hamiltonian $h$ and auxiliary parameter $\ell$ are defined by
\beq\label{piihell} 
h = -\tfrac{1}{2}p^2-pq^2
-\tfrac{z}{2}p+\ell q, \qquad 
\ell =\al +\tfrac{1}{2}. 
\eeq
Regarding the Hamiltonian $h=h(q,p,z)$ 
for Painlev\'e II, evaluated on a solution $q(z),p(z)$ 
of the system (\ref{hampii}), as a function of the independent variable $z$, Okamoto's tau function $\tau(z)$ is defined (up to an overall constant multiplier) by 
\beq\label{taudef}
h(z) = -\frac{\dd}{\dd z}\log \uptau(z).
\eeq 
The Painlev\'e property for the equation (\ref{pii}) is equivalent to the statement that the tau function $\uptau(z)$ is holomorphic. Indeed, the birational action of $S\circ T$ on the extended phase space  coordinates $(q,p,z,\ell)$ is  
\beq\label{birat}
S\circ T: \qquad \begin{cases}
    q & \mapsto \, -q +\frac{\ell}{p}\\
    p  &\mapsto \, -p-2(q-\frac{\ell}{p})^2-z \\
    z &\mapsto \, z \\
    \ell &\mapsto \, \ell+1 ,
\end{cases} 
\eeq
which preserves a contact 2-form $\Omega$, that is 
$$ 
(S\circ T)^*\Omega = \Omega, \qquad \Omega = \dd p \wedge \dd q -
\dd h\wedge \dd z, 
$$
and if the variables, Hamiltonians and tau functions are indexed 
by the parameter $\ell$, then 
$$ h_{\ell}=(S\circ T)^*(h_{\ell-1}) =h_{\ell-1}+q_{\ell}, $$
so that the solution of Painlev\'e II is 
given by the logarithmic derivative 
of a ratio of tau functions, as 
\beq\label{piitauq}
q_\ell (z)= \frac{\dd}{\dd z}\log \left(\frac{\uptau_{\ell-1}(z)}{\uptau_\ell(z)}\right), 
\eeq
while comparing (\ref{taudef}) with (\ref{piihell}) and using the fact 
that
$$
h_\ell '(z) = \frac{\partial}{\partial z} h_\ell (q,p,z), 
$$
the corresponding conjugate momentum variable 
is given by 
\beq\label{pltau}
p_\ell (z) = 2 \frac{\dd^2}{\dd z^2}\big( \log \uptau_\ell (z) \big). 
\eeq
Hence the solution of 
(\ref{pii}) has simple poles at the places where one of the adjacent tau functions 
$\uptau_{\ell-1}, \uptau_\ell$ has a zero, 
while $p_\ell$ has double poles where $\uptau_\ell$ has a zero. 

For the particular sequence of rational solutions in  
Table \ref{piirat}, a determinantal representation was found by Kajiwara and Ohta, by making a scaling reduction from the rational solutions of the KP hierarchy. 
The following is a slightly adapted version of the statement in \cite{ko}.   
\begin{theorem}\label{piiratthm}
For the values $\ell=n+\sfrac{1}{2}$ of the parameter in the Hamiltonian (\ref{piihell}), 
the Painlev\'e II equation (\ref{pii}) 
has rational solutions 
given by 
$$ 
q_{n+\sfrac{1}{2}} = \frac{\dd}{\dd z}\log \left(\frac{\uptau_{n-\sfrac{1}{2}}(z)}{\uptau_{n+\sfrac{1}{2}}(z)}\right), 
$$
where the tau function associated with $h_{n+\sfrac{1}{2}}$ is given by 
\beq\label{piidet} \uptau_{n+\sfrac{1}{2}}(z)
= \exp\left(-\frac{z^3}{24}\right) \, 
\left| \begin{array}{cccc}
\rp_{1}(z) & \rp_{3}(z) & \cdots & \rp_{2n-1}(z) \\ 
\rp_{0}(z) & \rp_{2}(z) & \cdots & \rp_{2n-2}(z) \\
\vdots & \vdots & \ddots & \vdots \\ 
\rp_{-n+2}(z) & \rp_{-n+4}(z) & \cdots & \rp_{n}(z)
\end{array}\right|
\eeq 
for integer $n\geq 1$, with 
$y_{n+\sfrac{1}{2}}=-y_{-n-\sfrac{1}{2}}$ for $n\leq 0$, and  polynomials $p_k(z)$ 
defined by the 
generating function 
\beq\label{pgenfn} 
\exp\left(z\la -\frac{4}{3}\la^3\right) = \sum_{k=0}^\infty \rp_k(z) \la^k, 
\qquad {\it  with}\quad  \rp_k(z)=0\quad {\it for}\quad k<0. 
\eeq
\end{theorem}
Observe that, from differentiating the above generating function with respect to $z$, it follows that the derivatives of the polynomial entries appearing in (\ref{piidet}) satisfy 
$$ \rp'_k(z) = \rp_{k-1}, $$
so that each row is the derivative of the one above it, and hence each determinant is a Wronskian. The sequence of tau functions begins with 
$$
e^{-\sfrac{z^3}{24}},e^{-\sfrac{z^3}{24}}z,
\frac{e^{-\sfrac{z^3}{24}}(z^3+4)}{3},\frac{e^{-\sfrac{z^3}{24}}(z^6+20z^3-80)}{45},
\frac{e^{-\sfrac{z^3}{24}}(z^{10}+60z^7+11200z)}{4725},
$$
for $n=0,1,2,3,4$. The exponential prefactor 
$e^{-\sfrac{z^3}{24}}$ cancels out from the ratio of tau functions in  (\ref{piitauq}), and the numerical prefactors make no difference to the logarithmic derivative, 
so that the rational solutions are determined by a sequence of monic polynomials whose degrees are triangular numbers, known as the 
Yablonskii-Vorob'ev polynomials 
(see \cite{kaneko2003coefficients} and references). In 
\cite{ko}, these polynomials are constructed by making a 
reduction from the well-known polynomial tau functions for  the KP hierarchy; but in section 2 we will show how the 
rational solutions of Painlev\'e II 
arise naturally from a sequence 
of Darboux transformations applied to 
Schr\"odinger operators, starting from a zero potential, in a construction due to Adler and Moser  \cite{adler1978class}. 
However, it turns out that these Darboux transformations are of a degenerate type, known as confluent, where the same eigenvalue is repeated in successive steps. In the confluent case,  the original formulation of Crum's theorem \cite{Crum} about iterated Darboux transformations does not apply.

Confluent Darboux transformations, and their  associated Wronskian representation under iteration, were studied systematically by 
Contreras-Astorga and Schulze-Halberg 
\cite{contreras2017recursive, 
schulze2013wronskian}. In the next section, 
we briefly review the theory of 
Darboux transformations, and the confluent case in particular, before the following section, where we discuss the 
relation with the work of Adler and Moser on 
the vanishing rational solutions of KdV, 
thereby explaining 
how this leads to determinantal formulae for the  Yablonskii-Vorob'ev polynomials via similarity reduction. 
In section \ref{ohya}, we describe the Ohyama polynomials, which correspond to a 
sequence of tau functions  
for the algebraic solutions of the 
Painlev\'e equation 
\begin{equation}\label{P3D7}
    \frac{\dd^2P}{\dd z^2}=\frac{1}{P}\left( \frac{\dd P}{\dd z} \right)^2-\frac{1}{z}\left( \frac{\dd P}{\dd z} \right) +\frac{1}{z}(2P^2-\beta)-\frac{1}{P} . 
\end{equation} 
The latter is referred to as the  
Painlev\'e III ($D_7$) equation, because it 
has surface type $D_7^{(1)}$; and it has symmetry type $(A_1^{(1)})_{|\mathbf{\al}|^2=4}$, with the additional suffix 
denoting 
a simple 
root having double the usual squared length: see Table 4 in \cite{sakai2001rational}. The rest of the paper is devoted to the describing the main 
result, namely 
an explicit Wronskian representation  for the 
algebraic solutions of (\ref{P3D7}), 
which is constructed by means of a sequence of confluent Darboux transformations applied 
to a Schr\"odinger operator that was obtained 
in \cite{hone1999associated} 
by reduction from the Lax pair for the Camassa-Holm equation  (see also \cite{barnes2022similarity}). 
The corresponding algebraic solutions of 
the  
Painlev\'e III ($D_7$) equation are rational functions in $z^{\frac{1}{3}}$, while the associated sequence of Ohyama polynomials are polynomials in $z^{\frac{2}{3}}$.  
The Riemann-Hilbert problem for 
these solutions was recently analysed and used to determine their asymptotic behaviour in \cite{buckingham2022algebraic}. Here our 
primary  concern is to present,  for the first time, a determinantal formula for these solutions,  
 which was identified as an open problem  in 
\cite{clarkson2003third}. The proof of our main result is given in section \ref{OhyamaSection}, with an associated 
generating function and 
its connection with Lax pairs being presented in section \ref{laxgen}. We end with some brief conclusions and an outlook on future work, while  some further technical details about intertwining operators and the associated 
algebraic structure for confluent Darboux transformations 
have been  relegated to an appendix.

\section{Confluent Darboux transformations}

The concept of the Darboux transformation was first proposed in 
\cite{DarbouxOriginal},  
as a covariance property of a 
general linear equation of second order, 
transforming any such equation 
into a new one with related coefficients. 
In the simplest case of an equation of Schr\"odinger type, 
written as 
\begin{equation}\label{schro0}
    \varphi''+(V_0+\la)\varphi= 0, 
\end{equation}
with eigenvalue $\la$ and potential $V_0(z)$ (strictly speaking, this is 
minus the potential in the context of quantum mechanics),  
the new equation obtained 
under the action of the Darboux transformation 
is
\begin{equation}\label{schro1}
    \tilde{\varphi}''+\left(\psi\frac{\dd^2}{\dd z^2}\Big(\frac{1}{\psi}\Big) +\la \right)=0,  
\end{equation} 
where 
the primes denote derivatives  
with respect to the independent variable $z$, 
and 
the new eigenfunction is 
\begin{equation} \label{basetransform}
    \tilde{\varphi}=\varphi'-\frac{\psi'}{\psi}\varphi , 
\end{equation}
 for some particular solution $\psi$ to the original 
Schr\"odinger equation (\ref{schro0}) with an arbitrary fixed choice of eigenvalue, $\mu$ say. Notably, (\ref{schro1}) 
can also be written as 
\begin{equation}\label{newschro}
\tilde{\varphi}''+(V_1+\la)\tilde{\varphi}= 0
\end{equation}
so that the transformation has produced an eigenfunction for an equation of the same type, but with 
a new potential $V_1(z)$ given by 
\beq\label{newpot} 
V_1=V_0+2(\log \psi)''.
\eeq 
\par

One of the most direct ways to 
understand the algebraic structure of the Darboux transformation is through factorization of the Schr\"odinger operator.    Upon 
introducing the first order operators 
\beq\label{1storder} 
L = \frac{\dd}{\dd z} - y, \qquad 
L^\dagger = -\frac{\dd}{\dd z} - y, \qquad \mathrm{with} \quad y = \frac{\dd}{\dd z} (\log\psi), 
\eeq
which, from (\ref{basetransform}), are  
such that 
\beq\label{kerL} 
L\psi = 0, \qquad L\varphi=\tilde{\varphi}, 
\eeq 
one has the factorizations 
$$ 
-L^\dagger L = \frac{\dd^2}{\dd z^2} +V_0 +\mu, \qquad 
-L L^\dagger  = \frac{\dd^2}{\dd z^2} +V_1 +\mu,
$$
with $V_1=V_0+2y'$ by  (\ref{newpot}), 
and then by 
(\ref{schro0}) and (\ref{kerL}) it follows that 
$$ 
L^\dagger\tilde{\varphi}=(\la -\mu)\varphi,  
$$
from which the new Schr\"odinger equation 
(\ref{newschro}) is obtained  by applying $L$ to both sides. Also, from (\ref{1storder}) it follows that 
\beq\label{psiinv}
L^\dagger \psi^{-1}=0 \implies 
-LL^\dagger \psi^{-1} = 
\left( 
\frac{\dd^2}{\dd z^2} +V_1 +\mu\right) \, \psi^{-1} = 0.
\eeq

An elegant extension 
of the above expressions, 
describing  the repeated action of successive Darboux transformations, was given by Crum \cite{Crum}, who provided a neat 
formulation of the overall 
transformation required to take the initial potential and solution %
from a given base system to one that is 
obtained by $n$ 
applications of the 
transformation defined by (\ref{basetransform}) and (\ref{newpot}), essentially reducing $n$ steps to 
a single step. 
The key is to take $n$ independent 
eigenfunctions 
$\psi_i$ for the original equation, with associated eigenvalues $\mu_i$, that is 
\begin{equation}\label{psisolns}
\psi_i''+(V_0+\mu_i)\psi_i=0, \qquad i=1,\ldots,n,     
\end{equation}
and consider their 
Wronskian 
$$
    Wr(\psi_1,\psi_2,...,\psi_n)=\det 
    \left(\frac{\dd^{i-1}\psi_j }{\dd z^{i-1}}\right)_{i,j=1,\ldots,n}.  
$$ 
Then the solution $\varphi_n$ of the Schr\"odinger equation 
\beq\label{ndtschro} 
\varphi_n''+(V_n+\la)\varphi_n=0
\eeq 
obtained by iterating the Darboux transformation 
$n$ times, starting from an initial solution 
$\varphi=\varphi_0$, is given by a ratio of Wronskians, namely 
\begin{equation}\label{ndt}
    \varphi_n=\frac{Wr(\psi_1,\psi_2,...,\psi_n,\varphi)}{Wr(\psi_1,\psi_2,...,\psi_n)},
\end{equation}
while the potential in (\ref{ndtschro}) 
is given in terms of the original one by  
\begin{equation}\label{ndtpot}
V_n =V_0+ 2\big(\log Wr(\psi_1,...,\psi_n)  \big)''  .
\end{equation}

Note that, in order to have distinct potentials $V_i$ and non-vanishing 
Wronskians $Wr(\psi_1,\psi_2,...,\psi_i)$ at each stage $0\leq i\leq n$,   Crum's description requires that the eigenfunctions 
$\psi_i$ chosen should have distinct eigenvalues $\mu_i$. However, it turns out that an analogous 
Wronskian description can be found in the case of so-called confluent Darboux transformations, where the new eigenfunction introduced at each stage has the same eigenvalue. 
 This confluent case is the one that is relevant  
 to the repeated application of the  BTs   (\ref{piibts}) for Painlev\'e II, 
 and also applicable to the BTs for the 
 Painlev\'e III ($D_7$) equation 
 (\ref{P3D7}),  
 which are the main object of our study here,  
 ultimately giving rise to a Wronskian representation for the Ohyama polynomials. 
 Before proceeding with the later, we 
 first summarize some of the results on 
 confluent Darboux transformations 
 from \cite{schulze2013wronskian} and 
 \cite{contreras2017recursive}.

In the confluent case, the entries of the Wronskian that produces the $n$th iteration of the Darboux transformation are no longer simple eigenfunctions of the original potential $V_0$, but instead are replaced by 
a sequence of generalized eigenfunctions 
$\psi_i$ satisfying a so-called Jordan chain, which means that, in particular,  
for each $i$ the condition 
\begin{align} \label{JordanChain}
\left(\frac{\dd^2}{\dd z^2}+V_0+\mu \right)^{i}\psi_{i}&=0, \qquad i=1,2, \ldots 
\end{align}
must hold, where 
$\mu$ is the common eigenvalue shared by all the solutions used in the application of the Darboux transformation at each step $i$. The 
condition (\ref{JordanChain}) 
implies that 
$$ 
\left(\frac{\dd^2}{\dd z^2}+V_0+\mu \right)\psi_{i}\in 
\mathrm{ker}\left(\frac{\dd^2}{\dd z^2}+V_0+\mu \right)^{i-1},  
$$
but in order to iterate the Darboux transformation,  
the Jordan chain 
should satisfy 
the stronger condition 
\beq\label{Jchain}
\left(\frac{\dd^2}{\dd z^2}+V_0+\mu \right)\psi_{i}=\psi_{i-1}, 
\eeq
up to an overall non-zero constant multiplier,  
corresponding to a choice of normalization for the eigenfunctions (and possible 
addition to right-hand side of  linear 
combinations of 
$\psi_k$ for $k<i-1$, has been suppressed, without loss 
of generality).
It is convenient to set $\psi_0=0$ so that the latter relation is valid for all $i\geq 1$. 
It can then be shown by induction 
(see the proof of Theorem 1 in \cite{schulze2013wronskian}) that, for each 
integer $n\geq 0$, after $n$ steps the function  
\begin{equation} \label{Jordanchi}
    \phi_n = \frac{Wr(\psi_1,\dots,\psi_{n+1})}{Wr(\psi_1,\dots,\psi_{n})}, 
\end{equation}
is a solution to 
\beq\label{ConfluentPotential}  \phi_n''+(V_n+\mu)\phi_n=0, \qquad 
\text{where}\quad V_n = V_0 + 
2\big(\log Wr(\psi_1,\dots,\psi_{n})\big)''. 
\eeq 
Furthermore, for all $n\geq 1$, $\phi_{n-1}^{-1}$ is another independent eigenfunction 
of the operator with potential $V_n$:
\beq\label{phiindept}
 (\phi_{n-1}^{-1})''+(V_n+\mu)\phi_{n-1}^{-1}=0, \qquad 
\text{with} \quad 
Wr(\phi_{n-1}^{-1},\phi_n)=1. 
\eeq 
It will be convenient to use $\theta_n$ to  denote the Wronskian of the first $n$ generalized eigenfunctions, 
and then it follows directly 
from \eqref{Jordanchi} that
\begin{equation}
   \theta_n= Wr(\psi_1,\dots,\psi_{n})=\prod^{n-1}_{j=0}\phi_j, \qquad \mathrm{for} \quad n\geq 1, 
\end{equation}
with $\theta_0=1$ corresponding to the ``empty'' Wronskian.   
\par

For completeness,  some further details about Jordan chains and associated Wronskian formulae have been included in the appendix.

\section{Yablonskii-Vorob'ev polynomials via   
Adler and Moser} 

In this section, 
we describe a simple example of a Jordan chain of generalized eigenfunctions that produces a sequence of Schr\"odinger operators, 
appearing  in the work of  Adler and Moser \cite{adler1978class}, who 
used iterated Darboux 
transformations to construct the 
rational solutions of 
the Korteweg-deVries 
(KdV)  equation, which is the PDE  
\beq\label{kdv} 
V_t=V_{xxx}+6VV_x. 
\eeq
By a direct  argument, they were able to 
construct a Wronskian representation for these solutions from scratch, by matching them to the polynomial solutions of the recurrence relation 
\begin{equation} \label{AMTheta}
    \theta_{k+1}'\theta_{k-1}-\theta_{k+1}\theta_{k-1}'=\theta_k^2 \quad \quad \text{for } k\geq 1,  
\end{equation}
starting with  $\theta_0=1$, $\theta_1=x$, where (above and 
in the most of the rest of this section) the $'$ denotes differentiation by $x$. 
(Note that, compared with 
\cite{adler1978class}, 
we have removed a factor $(2k+1)$ 
from the right-hand side above.) 
In fact, these 
polynomials and the recurrence (\ref{AMTheta}) 
were first discovered by Burchnall and Chaundy in the 1920s: see \cite{veselov2015burchnall} 
and references for further details. 
However, here we show how 
the Wronskian expressions for the 
polynomials $\theta_k$ are obtained  
immediately by applying the theory of  confluent Darboux transformations to a specific Jordan chain. In addition, 
we explain how the rational 
solutions of Painlev\'e II, as in  
Table \ref{piirat}, arise as a 
special case of  this construction, 
by taking a scaling similarity reduction 
of the KdV equation. 

The solutions of the recurrence 
(\ref{AMTheta}) admit the freedom 
to replace $\theta_{k+1}\to \theta_{k+1}+c \theta_{k-1}$, for an arbitrary integration 
constant $c$, since each 
$\theta_{k+1}$ can be found from 
the 
previous two terms in the sequence by 
integrating 
$$
\frac{\dd }{\dd x} \left(\frac{\theta_{k+1}}{\theta_{k-1}}
\right) = \left(\frac{\theta_{k}}{\theta_{k-1}}
\right)^2   , 
$$
although it is by no means 
obvious that integration of the 
rational function on the right-hand side 
should automatically lead to a new 
polynomial $\theta_{k+1}$ at each stage. 
Adler and Moser showed directly that, for each $k\geq 0$, the pair of independent functions   
\begin{equation}\label{phipair} 
    \phi_k=\frac{\theta_{k+1}}{\theta_k}, \quad \quad (\phi_{k-1})^{-1}=\frac{\theta_{k-1}}{\theta_k},
\end{equation}
lie in the kernel of 
the Schr\"odinger operator 
\beq\label{VkKdV}
\frac{\dd^2}{\dd x^2}+V_k, 
\qquad \mathrm{with} \quad  
V_k = 2\frac{\dd^2}{\dd x^2}\left(\log \theta_k\right), 
\eeq 
where it is consistent to take $\theta_{-1}=1$ 
so that (\ref{phipair}) with $k=0$ yields the two independent solutions $\phi_1=x$, 
$\phi_0^{-1}=1$ of the Schr\"odinger equation 
with the initial potential $V_0=0$. 
The Burchnall-Chaundy relation (\ref{AMTheta}) 
is equivalent to the normalization 
of the Wronskian of the pair of solutions  
(\ref{phipair}): 
$$ 
Wr(\phi_{k-1}^{-1},\phi_k)=1. 
$$ 

To apply a sequence of confluent Darboux 
transformations with repeated eigenvalue $\mu=0$, starting from the potential $V_0=0$,  
one could start from any solution 
$\phi_0=ax+b$, but by using the freedom to rescale 
and translate the independent variable, we take $\phi_0=x$.   
Then the  Jordan chain associated with this initial potential $V_0=0$ and zero eigenvalue 
$\mu=0$ is particularly simple: 
the condition (\ref{JordanChain}) becomes 
\begin{align}
    \left(\frac{d^2}{dx^2} +V_0+\mu\right)^{k}\psi_k&=\frac{\dd^{2k}\psi_k}{\dd x^{2k}}=0, 
\end{align}
which implies that $\psi_{k}$ is a polynomial 
of degree at most $2k-1$. 
For consistency with (\ref{Jordanchi}) we can fix  
$\psi_1=x$, 
and then from (\ref{Jchain}) all the  other generalized eigenfunctions $\psi_k$ are 
found recursively by integrating the relation 
\beq\label{2nd}
\psi_k''=\psi_{k-1},  
\eeq 
which yields  
\begin{equation} \label{AMPsi}
    \psi_k=\frac{x^{2k-1}}{(2k-1)!}+\sum^{k-2}_{i=0}c_{k-i}\frac{x^{2i}}{(2i)!},
\end{equation}
for a set of arbitrary constants $c_i$. 
Observe that, despite performing two integrations at each step, the formula (\ref{AMPsi}) only contains half as many constants as one would expect, and only even powers of $x$ are added to the initial term of odd degree, because we have exploited the 
freedom to  subtract from $\psi_k$ any multiples of $\psi_j$ with $j<k$, 
which makes no difference to the sequence of Wronskians $\theta_k$. Upon assigning a weight 
$2i-1$ to each constant $c_i$, we see that each generalized eigenfunction $\psi_k$ is a weighted homogeneous polynomial of degree  $2k-1$, and hence $\theta_k$ has weight 
$\sum_{i=1}^k(2i-1) = \tfrac{1}{2}k(k+1)$. Up to rescaling, the $c_i$ are equivalent to the constants denoted $\tau_i$ in \cite{adler1978class}, and they correspond to the times of the KdV hierarchy. 
Hence we arrive at
the following result. 

\begin{theorem}\label{burchau}
The sequence of polynomials  
given by $\theta_{-1}=1=\theta_0$ and the Wronskians 
\beq\label{bcpolys}
\theta_k(x,c_2,c_3,c_4,\ldots)=Wr(\psi_1,\dots,\psi_k) 
\qquad 
\mathrm{for}\quad k\geq 1
\eeq 
 with entries    
$\psi_i$ defined by 
\eqref{AMPsi}, 
satisfies the Burchnall-Chaundy relation  
\eqref{AMTheta}. Moreover, if we identify 
$c_2= 4t$, then each of the 
potentials 
$$ 
V_k = 2\, \frac{\dd^2}{\dd x^2}(\log \theta_k) 
$$
is a rational solution of the KdV equation (\ref{kdv}), and similarly for each $c_i$, $i>2$ 
identified as the higher time of weight $2i-1$ in the KdV hierarchy. 
\end{theorem}

The only part of the above result that we have not discussed 
so far is the dependence on time $t$, and the other times in the KdV hierarchy. This is best 
understood by considering the Lax pair for the KdV hierarchy, which corresponds to an isospectral evolution of the Schr\"odinger 
operator $\cal L$, defined as the compatibility condition 
\beq\label{laxkdv}
\partial_t {\cal L}  = [{\cal M},{\cal L}] 
\eeq for the linear system 
\begin{align}
({\cal L}+\la)\phi & = 0,\label{lax1} \\ 
\partial_t\phi &={\cal M}\phi, \label{lax2} 
\end{align}
where 
$$ 
{\cal L}= \frac{\dd^2}{\dd x^2}+V, 
\qquad {\cal M} = 4{\cal L}^{\sfrac{3}{2}}_+ = 4\frac{\dd^3}{\dd x^3}+6 V\frac{\dd}{\dd x}+3 V .
$$
(Despite the fact that we are now considering partial derivatives, we reserve the ordinary 
derivative symbol for the distinguished variable $x$.) It turns out that the Darboux transformation acts  covariantly  
not only on the Schr\"odinger equation (\ref{lax1}), which is the $x$ part of the linear system, but also on the $t$ evolution of the wave function $\phi$, given by (\ref{lax2}). 
This leads to the familiar 
result that 
the Darboux transformation for the 
Schr\"odinger operator induces a 
BT on the KdV equation, sending solutions to solutions. (For a detailed discussion of this property of the Darboux transformation, and its extension to other equations of Lax or zero curvature type, and associated  integrable PDEs, see \cite{matveevsalle}.)     
Similarly, for each $i$,  up to scaling we can identify $c_i$ with a  time variable $t'$, and the corresponding member of the KdV hierarchy is given by 
the Lax flow 
$$
\partial_{t'} {\cal L}  = [{\cal M}_i,{\cal L}]
$$
where ${\cal M}_i={\cal L}^{\sfrac{(2i-1)}{2}}_+$, 
which is the compatibility condition of the 
Schr\"odinger equation (\ref{lax1}) with the 
time evolution 
$$ 
\partial_{t'}\phi = {\cal M}_i\phi. 
$$

Starting from the simplest (vacuum) solution 
$V_0=0$, a single Darboux transformation with 
eigenvalue zero produces 
the stationary rational solution, which is given 
by the second logarithmic derivative 
$$ 
V_1 = 2 
(\log\theta_1 )'' 
=-\frac{2}{x^2}.
$$
(In contrast, applying a Darboux transformation with a non-zero eigenvalue 
produces a soliton solution from the vacuum.) 
Under the action of the confluent Darboux transformation, the next solution is 
obtained from  (\ref{VkKdV}) for $k=2$, that  
is 
\begin{equation}
   \theta_2= Wr(\psi_1,\psi_2)= \begin{vmatrix}
x & \frac{x^3}{6}+c_2 \\
1 & \frac{x^2}{2}
\end{vmatrix} =\frac{x^3}{3}-c_2,
\end{equation}
so with $c_2=4t$ the corresponding solution of 
(\ref{kdv}) is 
\beq\label{V2} 
V_2 = 
2\left(\log\Big(\frac{x^3}{3}-c_2\Big) \right)'' = 
-\frac{6x(x^3 + 24 t)}{(x^3 - 12 t)^2}. 
\eeq  

To see how the rational solutions of 
Painlev\'e II can be deduced from this construction, it is sufficient to 
consider scaling similarity solutions of KdV, 
which take the form 
\beq \label{kdvsred} 
V(x,t) = 
(-3t)^{-\sfrac{2}{3}} \big( p(z) +\tfrac{z}{2}\big) , \qquad 
\mathrm{with} \quad z=x(-3t)^{-\sfrac{1}{3}} . 
\eeq
Upon substituting the above expression into 
(\ref{kdv}), the 3rd order PDE reduces to the 
2nd order ODE  
\beq\label{p34}
\frac{\dd^2 p}{\dd z^2}=\frac{1}{2p}\left(\frac{\dd p}{\dd z}\right)^2 
-2p^2-zp-\frac{\ell^2}{2p}, 
\eeq 
where the coefficient $\ell^2$ arises as an 
integration constant. The latter ODE is called the Painlev\'e XXXIV equation \cite{ince}, and it is precisely the equation satisfied by the 
conjugate momentum variable $p$ when $q$ is 
eliminated from the pair of 
Hamilton's equations (\ref{hampii}). 
Thus, letting primes now denote derivatives 
with respect to $z$, 
\beq\label{pqform} 
q = \frac{p'+\ell}{2p} 
\eeq 
satisfies the Painlev\'e II equation 
(\ref{pii}) with parameter $\al=\ell-\tfrac{1}{2}$ whenever $p$ is a solution of (\ref{p34}); and conversely, 
whenever $q$ is a solution of Painlev\'e II, 
it follows that 
\beq\label{miura}
p=-q'-q^2-\tfrac{z}{2}
\eeq 
satisfies Painlev\'e XXXIV with $\ell=\al+\tfrac{1}{2}$. The latter formula also 
arises by reduction of the Miura 
transformation 
\beq \label{mmap} 
V=-y_x-y^2, 
\eeq 
which maps solutions of the modified 
Korteweg-deVries (mKdV) equation 
\beq\label{mkdv} 
y_t=y_{xxx}-6y^2y_x
\eeq 
to solutions of the KdV equation (\ref{kdv}). 
Indeed, if one takes the similarity reduction 
\beq \label{mkdvsred}
y(x,t)=(-3t)^{-\sfrac{1}{3}} q(z) 
\qquad 
\mathrm{with} \quad z=x(-3t)^{-\sfrac{1}{3}}, 
\eeq 
then the mKdV equation reduces to 
Painlev\'e II, 
while by substituting (\ref{mkdvsred}) into the Miura formula (\ref{mmap}), it is clear 
that $V(x,t)$ 
is a similarity solution of KdV, being of the 
form (\ref{kdvsred}) with $p$ given in terms of $q$ by (\ref{miura}). 

Under the similarity reduction (\ref{kdvsred}), 
the trivial potential $V_0=0$ corresponds to the solution $p=-\tfrac{z}{2}$ of (\ref{p34}) 
with $\ell=\sfrac{1}{2}$, while $V_1$ 
corresponds to 
$$ 
p=-\frac{2}{z^2} - -\frac{z}{2}, 
\qquad \ell=\sfrac{3}{2},  
$$ 
and from $V_2$ as in (\ref{V2}), the scaling reduction gives 
$$ 
p=-\frac{6z(z^3-8)}{(z^3+4)^2}  -\frac{z}{2}, 
\qquad \ell=\sfrac{5}{2}. 
$$
These solutions all have the correct 
weighted homogeneity 
under the scaling $x\to \gamma x$, $t\to\gamma^3 t$ of the KdV independent variables, but the other KdV solutions 
$V_k$ for $k>2$ do not have 
the right scaling behaviour unless we fix 
the higher time parameters $c_i$ to be zero for $i>2$. After making this adjustment, the 
fact that each of the  $\theta_k$ has homogeneous degree $\tfrac{1}{2}k(k+1)$ means that 
they can be rescaled to give polynomials in the similarity variable $z$.

\begin{lemma} After setting 
$c_2=4t$ and $c_i=0$ for $i>2$, the polynomial 
solutions of the Burchnall-Chaundy relation 
satisfy the scaling property 
\beq\label{scaling}
\theta_k(x,c_2,0,0,\ldots) = (-3t)^{\sfrac{k(k+1)}{6}}\, \theta_k \big(z,-\tfrac{4}{3},0,0,\ldots\big).
\eeq 
\end{lemma} 

If we now compare the Jordan chain 
for these scaling similarity solutions to the iterated action of the BT (\ref{birat})
for Painlev\'e II, we find that they match up 
exactly. To see this, note that for any fixed parameter $\ell$, since (\ref{p34}) is invariant under $\ell\to-\ell$, the formula (\ref{pqform}) relates two solutions of Painlev\'e II to the same solution of 
Painlev\'e XXXIV, by writing
\beq\label{pqaddn}
q_{\ell}=\frac{p_\ell '+\ell}{2p_\ell}, \qquad 
q_{-\ell}=\frac{p_\ell '-\ell}{2p_\ell}. 
\eeq 
Upon adding these two equations, substituting 
for $q_{\pm \ell}$ in terms of the Okamoto tau function $\uptau_\ell = \uptau_{-\ell}$ 
via (\ref{piitauq}), and integrating the logarithmic derivative that 
appears on both sides, we find that 
$p_\ell$ is given by the ratio 
\beq\label{ptauratio} 
p_\ell = C_\ell \, \frac{\uptau_{\ell -1}\uptau_{\ell +1}}{\uptau_{\ell}^2}, 
\eeq 
for some normalization constant $C_\ell$. 
On the one hand, 
comparing the above expression with (\ref{pltau}) yields an equation for the sequence of tau functions, 
\beq\label{todal} 
C_\ell \, \uptau_{\ell -1}\uptau_{\ell +1} = 2 
\big(\uptau_{\ell }\uptau_{\ell}'' - (\uptau_{\ell }')^2\big), 
\eeq  
which is a bilinear form of the Toda lattice.  
On the other hand,  subtracting one of the two equations (\ref{pqaddn}) from the other, 
then substituting for $y_{\pm \ell}$ using (\ref{piitauq}) and for $p_\ell$ using 
(\ref{ptauratio}), produces the 
Burchnall-Chaundy relation in the modified 
form 
\beq\label{bcmod}
\uptau_{\ell+1}'\uptau_{\ell-1} - 
\uptau_{\ell+1}\uptau_{\ell-1}' 
= -\ell\, C_\ell \, \uptau_\ell^2 .
\eeq 
Furthermore, when we restrict to the sequence of rational solutions with parameters $\ell=n+\sfrac{1}{2}$, and compare  (\ref{kdvsred}) with 
(\ref{scaling}), for each integer $n$  we see that the ratios 
$\uptau_{n+\sfrac{3}{2}}/\uptau_{n+\sfrac{1}{2}}$, $\uptau_{n-\sfrac{1}{2}}/\uptau_{n+\sfrac{1}{2}}$ provide a pair of independent 
eigenfunctions for the potential  
$$
p_{n+\sfrac{1}{2}} +\frac{z}{2} = 2 \frac{\dd^2}{\dd z^2} \Big(\log 
\theta_n \big(z,-\tfrac{4}{3},0,0,\ldots\big)
\Big) 
=2 \frac{\dd^2}{\dd z^2}\Big( \log \uptau_{n+\sfrac{1}{2}} (z) \Big) +\frac{z}{2}. 
$$
Then we can fix the normalization 
$- (n+\sfrac{1}{2})
C_{n+\sfrac{1}{2}}=1$ to match 
(\ref{bcmod}) with (\ref{AMTheta}), and note that this relation is invariant under rescaling all the tau functions by the same factor $e^{-\sfrac{z^3}{24}}$. Hence, from the above lemma, we have   

\begin{corollary}\label{yvco}
Up to normalizing constants, the Yablonskii-Vorob'ev polynomials are obtained from the Burchnall-Chaundy polynomials 
(\ref{bcpolys}) by replacing $x\to z$, $c_2\to -\tfrac{4}{3}$, and setting $c_i=0$ for all $i>2$. Moreover, they are related to  
the Okamoto tau functions by 
$$
\theta_{n}(z,-\tfrac{4}{3},0,0,\ldots) = \exp\left(\frac{z^3}{24}\right)\, \uptau_{n+\sfrac{1}{2}}(z). 
$$
\end{corollary}

An alternative route to the Yablonskii-Vorob'ev polynomials,  
and the one taken in \cite{ko}, 
is to start from the rational solutions of the KP hierarchy, then reduce these to the rational solutions of KdV, and finally make the similarity reduction to 
the corresponding solutions of Painlev\'e XXXIV/Painlev\'e II. 
The polynomial tau functions $\tau=\tau_Y(\underline{t})$ of the KP hierarchy are the Schur functions 
$$ 
\tau_Y (\underline{t}) 
= Wr (\rp_{j_n}, \rp_{j_{n-1}+1},\ldots, \rp_{j_1+n-1}), 
$$
associated with a Young diagram $Y$ defined by integers $j_1\geq j_2\geq \cdots \geq j_n$,  where $\underline{t}=(t_1,t_2,t_3,\ldots)$ is the sequence of KP times, 
and $\rp_j$ are the elementary Schur polynomials, with 
generating function 
\beq\label{schur} 
\exp\left(\sum_{i=1}^\infty t_i \la^i\right) =\sum_{j=0}^\infty \rp_j(\underline{t}) \la^j,  
\eeq  
which satisfy 
\beq\label{pderiv}
\partial_{t_i} \rp_j = \rp_{j-i}
\eeq 
and the infinite hierarchy of symmetries of the heat equation, that is 
$$
\partial_{t_k}\rp_j = \partial_x^k \rp_j$$
(with $t_1=x$). (For more details on the KP hierarchy and its 
solutions, see \cite{sato}.) 
The reduction from KP to KdV requires that all the even times should be discarded, so that only dependence on the odd times $t_1,t_3,t_5,\ldots$ remains, and the tau functions which survive are those that satisfy 
$$ 
\partial_{t_{2i}}\tau = 0, 
$$
which in the case of polynomial solutions requires that only the Schur functions with  triangular Young diagrams 
should remain, namely 
\beq\label{tri} 
\tau_{Y,KdV}(\underline{t}) = Wr (\rp_{1}, \rp_{3},\ldots, \rp_{2n-1}) . 
\eeq 
By comparison with (\ref{bcpolys}), we see that these are precisely the Burchnall-Chaundy polynomials, when we  identify $t_1=x$, $t_3=c_2=4t$ and $t_{2i-1}=c_i$ for $i>2$, and $\rp_{2i-1}=\psi_i$, with the Jordan chain condition  (\ref{2nd}) being a 
particular consequence of 
the general derivative property (\ref{pderiv}). 
Then under the scaling similarity reduction, the entries of the determinant (\ref{piidet}) and the generating function (\ref{pgenfn}) arise from (\ref{schur}) by 
replacing $t_1\to z$, $t_3\to -\tfrac{4}{3}$, and all other $t_i\to 0$. In a similar manner, for the PII hierarchy, which arises by taking scaling similarity reductions of the higher flows of the mKdV hierarchy (see e.g. \cite{chj}), one can obtain   
the rational 
solutions by replacing $x=t_1\to z$, fixing a non-zero value of the appropriate time $t_{2i-1}$, and setting all the other times to $0$.

\section{Ohyama polynomials 
and BTs for  Painlev\'e III ($D_7$)}\label{ohya} 


The Ohyama polynomials $\rho_n(s) $ are a sequence of polynomials 
defined recursively by the relation
\begin{equation}\label{ohrec}
(s+n)\rho_{n}^2-2s\rho_{n}\ddot{\rho}_{n}
+2s(\dot{\rho}_{n})^2-2\rho_{n}\dot{\rho}_{n} = 
\begin{cases}
    \rho_{n+1}\rho_{n-1} & \text{for $n$ odd,}\\
    s\rho_{n+1}\rho_{n-1}& \text{for $n$ even,}
\end{cases} 
\end{equation}
for $\rho_0=\rho_{\pm1}=1$, where the dots denote  differentiation with respect to the variable $s$, and $n\in\Z$. Despite it not being obvious from the form of this relation, it has been proven  \cite{ohyama2006studies} that each  $\rho_n$ is a monic polynomial  
in $s$, with integer coefficients, and $\rho_n(0)\neq0$.  As we shall see, these polynomials are in direct  correspondence with 
the algebraic solutions of a special  
case  of the Painlev\'e III equation, 
given by (\ref{P3D7}), 
arising as particular solutions when the parameter $\be$ therein is an even integer.   
As such, they play an analogous role  to the 
Yablonskii–Vorob’ev polynomials for Painlev\'e II, as in the previous section, and to other families of polynomials like the Umemura polynomials \cite{kajiwara1999umemura},  and the Okamoto polynomials \cite{clarkson2003third, okamoto}, which  are associated with another family of  
rational solutions of Painlev\'e III, 
and rational solutions of Painlev\'e IV, 
respectively. However, unlike these other polynomial families, until now,  no Wronskian or other determinantal 
representation was previously known for the Ohyama polynomials. In this section, we outline 
how such a 
representation arises from Darboux transformations.

The 
Painlev\'e III ($D_7$) equation, as in (\ref{P3D7}), can be derived from the system of Hamilton's equations  
\beq\label{hampiii} 
\begin{array}{rcrl}
zQ' & = & Q(2PQ-\ka)-z=& \frac{\partial \rh}{\partial P}, \\
zP' & = & -P(2PQ-\ka)+z =& -\frac{\partial \rh}{\partial Q},
\end{array}
\eeq 
where the prime denotes the $z$ derivative, and 
\beq\label{piiih}
\rh = Q^2P^2-\ka \,QP-z(Q+P), \qquad \ka = \be+1. 
\eeq
By eliminating $Q$ from the system, $P$ is found to satisfy the equation  (\ref{P3D7}), that is 
$$ 
P'' 
=\frac{1}{P}\left( P'\right)^2-\frac{1}{z}\left( P' \right) +\frac{1}{z}(2P^2-\be)-\frac{1}{P}
$$ 
while if $P$ is  eliminated instead then $Q=P_+$  is found to satisfy the same ODE but 
with $\be\to \be +2$, and by reversing the roles of $P$ and $Q$ one 
can shift $\be$ down by 2, leading to 
a BT for (\ref{P3D7}) together with its inverse, 
namely  the pair of transformations 
\begin{equation} \label{Backlund}
    P\mapsto P_{\pm}=\frac{z(\mp P'+1)}{2P^2}+\frac{(\pm1+\beta)}{2P}, \qquad \beta\mapsto \beta\pm 2. 
\end{equation}
The forward shift can be written as a birational transformation $T_+$ acting on the extended phase space 
with coordinates $(Q,P,z,\ka)$, given by 
\beq\label{piiibir}
T_+: \qquad \begin{cases}
    Q & \mapsto \, -P +\frac{\ka+1}{Q}+\frac{z}{Q^2}\\
    P  &\mapsto \, Q \\    z &\mapsto \, z \\
    \ka &\mapsto \, \ka+2 ,
\end{cases} 
\eeq
preserving a contact 2-form $\Omega$, namely  
$$ 
 T_+^*\Omega = \Omega, \qquad \Omega = \dd P \wedge \dd Q -
\frac{1}{z}\dd \rh\wedge \dd z, 
$$
and there is a similar set of expressions 
defining the inverse $T_-=T_+^{-1}$. 
Moreover, if the sequence of Hamiltonians 
obtained under the iterated action of $T_+$ is  indexed 
by the parameter $\ka$, then 
\beq\label{hshift} 
\rh_{\ka}= T_+^*(\rh_{\ka-2}) =\rh_{\ka-2}-\frac{z}{P}-\ka+1.
\eeq
(For the full set of affine Weyl group symmetries of the 
Painlev\'e III ($D_7$) equation, see \cite{ohyama2006studies}.)

A deeper insight into the structure of the BT $T_+$, 
and the understanding of its connection 
with confluent Darboux transformations,  
was achieved  due to the investigations in 
 \cite{hone1999associated}, which were further 
 clarified in \cite{barnes2022similarity}, 
 where it was shown that the 
 Painlev\'e III ($D_7$) equation 
 (\ref{P3D7}) arises as a similarity 
 reduction of the Camassa-Holm equation 
\beq\label{ch}
u_t - u_{xxt}+3uu_x =uu_{xxx}+2u_xu_{xx}. 
\eeq 
Although the full details are somewhat involved 
(cf.\ Theorem 3.2 in \cite{barnes2022similarity}), the similarity solutions 
of the PDE 
(\ref{ch}) can be specified in parametric form by the hodograph 
transformation 
\beq\label{hodo}
u(x,t) = \frac{1}{2t}\left(\frac{z}{P(z)}+\be\right) 
, \quad  
x-\frac{\be}{2}\log t =\log\left(\frac{\phi_-(z)}{\phi_+(z)}\right) +\mathrm{const}, 
\eeq 
where $P(z)$ is a solution 
of (\ref{P3D7}), and $\phi_{\pm }$ are two solutions 
of an associated 
Schr\"odinger equation 
\begin{equation} \label{painleveschrodinger}
    \left( \frac{\dd^2}{\dd z^2} + V \right)\phi_\pm = 0,
\end{equation}
with the potential $V=V(z)$ defined in terms of $P(z)$ by 
\begin{equation}\label{VP}
    V= -\frac{1}{4P^2}\left(\left(P' \right)^2-1 \right)+\frac{1}{2zP}\left(P'-2P^2+\beta \right),
\end{equation} 
and these two solutions of 
(\ref{painleveschrodinger}) 
are constrained 
by the 
requirement that their product is $P$ and their Wronskian is 1: 
\beq\label{psipmdet} 
P=\phi_+\phi_-, \qquad Wr(\phi_+,\phi_-)=1. 
\eeq 

As was explained in  \cite{hone1999associated}, 
the fact that the Camassa-Holm equation is related 
to a negative KdV flow  means that the BT $T_+$ for (\ref{P3D7}) can be obtained from a (confluent)  
Darboux transformation acting on the 
Schr\"odinger 
operator  in (\ref{painleveschrodinger}), 
where the potential $V$ is viewed as 
coming from a similarity 
solution of a member of  
the KdV hierarchy. (However, note that the 
variables $x,t$ in (\ref{ch}) are not directly related 
to $x,t$ in the previous section.) Indeed, observe that for any 
solution $P(z)$ of the    Painlev\'e III ($D_7$) equation, we can introduce a tau function $\si (z)$, 
which is defined by considering the quantity 
\beq\label{etadef}
\eta(z)=-P(z) - \frac{z}{2}\,V(z) 
\eeq 
and then a direct calculation  using (\ref{P3D7}) 
shows that 
\beq\label{etaV}
\eta'(z)=\frac{1}{2}V(z).  
\eeq
Thus, if 
$\si$ is related to $\eta$ by 
\beq\label{etareln} 
\eta(z) = \frac{\dd}{\dd z}\big(\log \si (z)\big), 
\eeq  
then it transpires that 
the corresponding KdV potential $V$ in (\ref{VP}) 
is given in terms of the same tau function by the 
standard relation 
\beq\label{Vsig}
V(z) = 2\frac{\dd^2}{\dd z^2}\big(\log \si (z)\big).   
\eeq 
Then in turn, by rearranging (\ref{etadef}) and 
substituting for $\eta$ and $V$ in terms of $\si$, 
it follows that $P$ is specified by the tau function   
according to the formula 
\beq\label{Psig}
P(z) = -\frac{\dd}{\dd z}
\left(z\frac{\dd}{\dd z}\big(\log \si (z)\big)\right).   
\eeq 
The connection between (\ref{Backlund}) 
and Darboux transformations is explained  by the next result. 

\begin{lemma}\label{dtlem} Given a solution $P(z)$ 
of the Painlev\'e III ($D_7$) equation, let 
\beq\label{ypmdef} 
y_\pm = \frac{P'\mp 1}{2P}. 
\eeq 
Then the BT $T_+$ and its inverse $T_-=T_+^{-1}$, 
as in (\ref{Backlund}), can be expressed as 
\beq\label{altbt}
P_\pm = P  - \frac{\dd}{\dd z}\big( zy_{\pm}\big). 
\eeq 
Moreover, the 
associated potential, defined by (\ref{VP}), 
can be written as 
\beq \label{Vfactors} 
V=-y_+'-y_+^2 = -y_-'-y_-^2 , 
\eeq 
and 
the corresponding action of $T_{\pm}$ on  $V$ is 
equivalent to a Darboux transformation, 
being given by 
\beq\label{dtformula}
V_\pm =T_{\pm}^*(V)=V+2y_{\pm}'. 
\eeq 
\end{lemma}
\begin{proof}
A direct calculation, using the 
Painlev\'e III ($D_7$) equation (\ref{P3D7}), shows that 
$$\begin{array}{rcl}
y_+' =\frac{\dd}{\dd z}\left(\frac{P'- 1}{2P}\right) 
& = & \frac{P''}{2P}-\frac{(P')^2}{2P^2}+\frac{P'}{2P^2} 
\\
& = & -\frac{1}{z}\left(\frac{P'+\be}{2P}-P\right)+
\frac{P'-1}{2P^2} \\
& = & -\frac{1}{z}\left(y_+ + \frac{\be+1}{2P}-P\right)+
\frac{P'-1}{2P^2},  
\end{array}
$$
and hence,  from the formula for $T_+$ in (\ref{Backlund}), 
$y_+'=z^{-1}(y_+ + P -P_+)$, which yields the $+$ case 
of (\ref{altbt}), and a similar 
calculation yields the $-$ case. 
Furthermore, for both choices of sign we find  
$$
-y_\pm'-y_{\pm}^2 = -\frac{P''}{2P}+\frac{(P')^2-1}{4P^2}, 
$$
and upon using the Painlev\'e III ($D_7$) equation once again, to remove the $P''$ term, 
we arrive at the formula (\ref{VP}) for $V$.  
Thus, since (\ref{Vfactors}) holds, we have two different 
factorizations of the Schr\"odinger operator with this potential: 
$$ 
\frac{\dd^2}{\dd z^2} + V = 
\left(\frac{\dd}{\dd z} +y_+\right) 
\left(\frac{\dd}{\dd z} -y_+\right) 
=
\left(\frac{\dd}{\dd z} +y_-\right) 
\left(\frac{\dd}{\dd z} -y_-\right).  
$$
Now let us introduce the corresponding pair of eigenfunctions with eigenvalue zero: 
$$
y_{\pm}=\frac{\dd}{\dd z} \log \phi_{\pm}\implies 
\left(\frac{\dd^2}{\dd z^2} + V \right) \phi_{\pm}= 0.  
$$
Then from the definitions of $y_\pm$ in terms of $P$, by adding we obtain  
$$
\frac{\dd}{\dd z} \log (\phi_+\phi_-) = y_+ + y_- =
\frac{\dd}{\dd z} \log P \implies P = \phi_+\phi_-, 
$$
where we have fixed an overall normalizing constant, 
while by subtracting we find 
$$ 
\frac{\dd}{\dd z} \log \left(\frac{\phi_-}{\phi_+}\right) = y_- - y_+ = \frac{1}{P}=\frac{1}{\phi_+\phi_-} 
\implies Wr(\phi_+,\phi_-)=1. 
$$
So we have verified the assertions in (\ref{psipmdet}). To relate  the BT $T_+$ to a Darboux transformation, it is helpful to determine its action on the quantity $\eta$  defined by (\ref{etadef}). After a slightly tedious computation,  
using the second equation of the system (\ref{hampiii}) to  substitute  
$P'=1-z^{-1}P(2PQ-\ka)$ in the expression (\ref{VP}), 
we find that $\eta$ is very closely linked with the 
Hamiltonian, being given by 
\beq\label{etaform} 
\eta= \frac{1}{2z}\Big(\rh+PQ+\tfrac{\ka(\ka-2)}{4}\Big). 
\eeq 
(In fact, $\eta$ is the same as the Hamiltonian denoted $H$ in \cite{ohyama2006studies}, where a different Hamiltonian structure is used.)    
Then by applying the BT to shift all the variables in   (\ref{hshift}), we have 
$$ 
T_+^*(\rh) = \rh-\frac{z}{Q}-\ka-1,  
$$ so from 
(\ref{piiibir}) we get 
$$ \begin{array}{rcl}
T_+^*(\eta) & = &
\frac{1}{2z}\left(\rh-\frac{z}{Q}-\ka-1
+Q\Big(-P+\frac{\ka+1}{Q}+\frac{z}{Q^2}\Big)+\frac{(\ka+2)\ka)}{4}\right) \\
& = & \frac{1}{2z}\Big(\rh -QP +\tfrac{\ka(\ka+2)}{4}\Big) \\
& = & \eta -\frac{1}{2z}\Big( 2PQ-\ka \Big). 
\end{array}
$$
By using the formula for $P'$ in (\ref{hampiii}) 
once again, we see that 
$y_+=-\tfrac{1}{2z}(2PQ-\ka)$, and so 
\beq\label{Tshift} 
T_+^*(\eta) =\eta +y_+. 
\eeq 
Hence,  by differentiating and applying the relation 
(\ref{etaV}), 
we see that 
$$ 
T_+^*(V) =V +2y_+' = V+2\frac{\dd^2}{\dd z^2}
\Big(\log\phi_+\Big), 
$$
which shows that $T_+$ corresponds to a 
Darboux transformation acting on the potential $V$. 
An analogous chain of reasoning produces the $-$ case 
of (\ref{dtformula}), which shows that $T_-=T_+^{-1}$ 
also corresponds to a 
Darboux transformation. 
\end{proof}

\begin{remark}\label{taurem}
We can introduce additional tau functions that appear under the action of the BT and its inverse, so that 
$$ 
T_{\pm}(\eta)=\frac{\dd}{\dd z} \log \si_{\pm}, 
$$
and then from  
(\ref{Tshift}) and the analogous formula for $T_-$, 
it follows that  
\beq\label{ypm}
y_\pm = \frac{\dd}{\dd z} \left(\log \Big(\frac{\si_\pm}{\si}\Big) \right). 
\eeq 
Also, the formula (\ref{etaform}) implies that 
$$ 
\eta = \frac{1}{2z}\Big(\rh + \tfrac{\ka^2}{4}\Big) - \frac{1}{2}y_+ ,$$ 
which means that the Hamiltonian can be   expressed 
in terms of tau functions as 
$$
\rh = z \frac{\dd}{\dd z} \Big(\log (\si\si_+) \Big)-\frac{\ka^2}{4}. 
$$
\end{remark}

We now introduce some recursion relations for the tau functions of the Painlev\'e III ($D_7$) equation, which 
will be key to understanding how the Ohyama polynomials, as well as the relation (\ref{ohrec}), appear in this context. 

\begin{proposition}\label{recprop} Suppose that 
$\si_-,\si,\si_+$ are three adjacent tau functions for the equation (\ref{P3D7}), connected via the action of 
the BTs $T_-$ and $T_+$.   Then these 
three tau functions are connected by a bilinear equation 
of Toda lattice type, namely 
\beq\label{sigtoda}
 \si_+\si_- =C\Big( z\big((\si')^2-\si\si'' \big) - \si\si' \Big), 
\eeq
and also by a Burchnall-Chaundy relation, that is 
\beq\label{sigbc} 
\si_-' \si_+ - \si_-\si_+' = C \, \si^2\, 
\eeq 
where $C$ is a non-zero constant. 
\end{proposition}
\begin{proof}
Either by adding the definitions (\ref{ypmdef}) of $y_\pm$ 
in Lemma \ref{dtlem}, and integrating the logarithmic derivative that appears on both sides after substituting 
in with the tau function expressions (\ref{ypm}), or by observing that the eigenfunctions  
in (\ref{psipmdet}) must be given by 
$$\phi_+ = K_+\, \frac{\si_+}{\si}, \qquad 
\phi_- = K_-\, \frac{\si_-}{\si} $$ 
for some non-zero constants $K_\pm$, 
we see that the solution $P$ of 
the Painlev\'e III ($D_7$) equation is given in terms of these three tau functions by the ratio 
\beq\label{Prat} 
P = C^{-1}\, 
\frac{\si_+\si_-}{\si^2}, 
\eeq 
for some constant $C=1/(K_+K_-)\neq 0$. Upon equating this with the logarithmic derivative  
in (\ref{Psig}), and clearing $\si^2$ from the denominator, the bilinear equation 
(\ref{sigtoda}) is obtained. 
If instead  the two definitions (\ref{ypmdef}) are 
subtracted one from the other, then we have 
$$ 
y_- - y_+ = \frac{1}{P}
$$ 
so that from substituting the  expressions (\ref{ypm}) on the left-hand side and the ratio (\ref{Prat}) on the 
right-hand side above, after clearing the 
denominator $\si_+\si_-$, this produces (\ref{sigbc}).  
\end{proof}    
\begin{remark} Both of the bilinear relations for the tau functions can be written somewhat more concisely by using the  Hirota  derivative $D_z$: the Toda-type relation is 
$$
\si_+ \si_- + C\Big(\tfrac{z}{2}\,D_z^2\, \si\cdot \si +\si \,\si' \Big) = 0,  
$$
while the Burchnall-Chaundy relation is 
$$ 
D_z\,  \si_-\cdot\si_+ = C\, \si^2.  
$$
\end{remark}

The algebraic solutions of the Painlev\'e III ($D_7$) equation arise for even 
integer values of the parameter $\be$, 
and are related to a family of ramp-like 
similarity solutions of the Camassa-Holm equation (see \cite{barnes2022similarity}). The PDE (\ref{ch}) 
has the elementary solution 
$$ 
u(x,t) = \frac{x}{3t}, 
$$
which is the same as the ramp solution for the inviscid 
Burgers' (Hopf) equation $u_t +3uu_x=0$, and is a 
particular similarity solution of the form 
(\ref{hodo}) with $\be=0$, corresponding to the solution 
$$ 
P = \left(\frac{z}{2}\right)^{1/3}
$$ 
of (\ref{P3D7}) for this value of $\be$.  
The action of one of the BTs (\ref{Backlund}) either raises or lowers  
the value of the parameter by 2 with each application, so by taking $P_0=\zeta$ as the seed solution, with 
$\zeta
=(z/2)^{1/3}$,  
a family of algebraic solutions is obtained for all even integer values $\beta=2n$, which we denote by 
$P_n$ for $n\in\Z$, and these are all rational functions   
of $\ze$; a few of these solutions are presented in the first row of Table \ref{FirstExamples}.     
Similarly, 
since the action of each BT or Darboux transformation increases or decreases $n$ by 1 at each step, 
it will be convenient to index all relevant 
quantities with this integer, while for the 
independent variables it will be necessary to switch between $\ze$ and the variables 
$$z=2\zeta^3, \qquad s=3\zeta^2. $$
The majority of the subsequent formulae are written most simply in terms of 
$\ze$, but in all of the Wronskians the derivatives  are taken with respect to the variable $z$, while the Ohyama polynomials defined by 
(\ref{ohrec}) are polynomials in $s$.

The precise connection between the Ohyama polynomials and the algebraic solutions of the Painlev\'e III ($D_7$) equation is that, up to multiplying by certain 
$n$-dependent gauge factors  and a change of independent variable, the polynomials $\rho_n(s)$ are 
equivalent to the tau functions $\si_n(z)$. 
Now to fix our notation: we will be considering the sequence of potentials $V_n=V_n(z)$ associated with the 
algebraic solutions $P_n(z)$. In terms of tau functions
$\si_n(z)$, we have 
$$ 
V_n = 2 \frac{\dd^2}{\dd z^2} (\log \si_n), 
\qquad 
P_n = -\frac{\dd}{\dd z}\left(z \frac{\dd}{\dd z}( \log \si_n)\right),
$$
while the action of the corresponding (confluent) 
Darboux transformations on the potentials can be written 
in terms of eigenfunctions $\phi_n, \tilde{\phi}_n$, 
so that 
$$ 
 V_{n+1}=V_{n}+2\frac{\dd^2}{\dd z^2}(\log\phi_{n}), 
 \qquad 
 V_{n-1}=V_{n}+2\frac{\dd^2}{\dd z^2}(\log\tilde{\phi}_{n}). 
$$
Comparing this with the notation used above in Lemma \ref{dtlem} and 
Proposition \ref{recprop}, we see that for $\ka=2n+1$ we have a solution $P=P_n$ of (\ref{P3D7}), with 
two adjacent solutions $T_\pm^*(P) = P_{n\pm 1}$ obtained 
by the action of the BT that shifts the parameter one step  up/down;  
the associated potential $V=V_n$ gets sent to a new potential $V_+=V_{n+1}$, via the Darboux transformation generated by the 
eigenfunction $\phi_+=\phi_n$ (with eigenvalue zero) satisfying  the Schr\"odinger equation  \eqref{painleveschrodinger},  
or is sent 
to the potential $V_-=V_{n-1}$, by the Darboux transformation generated by the 
eigenfunction $\phi_-=\tilde{\phi}_n$. 
Finally, our goal will be to construct the Jordan chain 
for the associated sequence of  confluent Darboux transformations applied to the initial potential 
\beq\label{V0}
V_0 = \frac{5}{36z^2} - \frac{1}{4(z/2)^{\frac{2}{3}} }
=\frac{(5-36\ze^4)}{144\ze^6}, 
\eeq
which will allow us to write each potential in the form  
\beq
 \label{Transform} 
 V_{n}=V_{0}+2\frac{\dd^2}{\dd z^2}(\log\theta_{n}) , 
\eeq 
where $\theta_n$ is a Wronskian built from generalized eigenfunctions $\psi_i$. 

Any consecutive triple of tau functions 
$\si_{n-1},\si_n,\si_{n+1}$  can be identified with the triple  
$\si_-,\si,\si_+$ in Proposition \ref{recprop}, 
and with a suitable choice of normalization we 
can write the Wronskians and eigenfunctions in terms of 
these tau functions, which leads to 
\beq
\label{thetasigma}  \theta_{n}=\frac{\sigma_{n}}{\sigma_0}, 
\eeq 
meaning that each of the Wronskians $\theta_n$ can be 
seen to be a kind of renormalized tau function,  
as well as the relations  
\beq 
\label{phitheta}   
\phi_{n}=\frac{\sigma_{n+1}}{\sigma_{n}}=\frac{\theta_{n+1}}{\theta_{n}}, \quad \mathrm{and} \quad 
    \tilde{\phi}_{n}=\frac{\sigma_{n-1}}{\sigma_{n}}=\frac{\theta_{n-1}}{\theta_{n}}=(\phi_{n-1})^{-1}.
\eeq 

Note that, for each $n$, the corresponding quantity $\eta=\eta_n$ is determined by substituting $P=P_n$ and $V=V_n$ into (\ref{etadef}), which means that, by (\ref{etareln}), the corresponding tau function is given by 
$$ 
\si_n = \exp \left(\int \eta_n\, \dd z\right ), 
$$
and hence is fixed up to an overall constant multiplier 
(coming from the implicit integration constant above). 
It is instructive to list here the first few tau 
functions, for $-3\leq n\leq 3 $: 
$$
\begin{array}{ll}
        \sigma_0=\zeta^{-\frac{5}{24}}e^{-\frac{9}{8}\zeta^4}, \quad & 
    \sigma_{\pm1}=3^{1/4}\zeta^{\frac{7}{24}}e^{-\frac{9}{8}\zeta^4\mp\frac{3}{2}\zeta^2}, \\ 
    \sigma_{\pm2}=(3\zeta^2\pm1)\zeta^{-\frac{5}{24}}e^{-\frac{9}{8}\zeta^4\mp3\zeta^2},\qquad &
    \sigma_{\pm3}=3^{1/4}(9\zeta^4\pm12\zeta^2+5)\zeta^{\frac{7}{24}}e^{-\frac{9}{8}\zeta^4\mp\frac{9}{2}\zeta^2}.
   \end{array}
$$ 
Note that, compared with 
\cite{barnes2022similarity} we have switched $n \rightarrow -n$, in order to be consistent with the convention used to define the Ohyama polynomials  by 
(\ref{ohrec}). Then, by the results of Proposition \ref{recprop}, this sequence of tau functions satisfies the bilinear Toda-type relation 
\begin{equation}
    \label{tauToda}
    \sigma_{n+1}\sigma_{n-1}+C_n \left(\tfrac{z}{2}D^2_z\sigma_{n}\cdot\sigma_{n}+\sigma_{n}\sigma_{n}'\right)=0,
\end{equation}
where the constant $C$ from (\ref{sigtoda}) 
(which depends on the choice of scaling) 
is allowed to depend on $n$, while from (\ref{sigbc}) we also have the Burchnall-Chaundy relation 
\beq \label{taubc}
\si_{n-1}' \si_{n+1} -  \si_{n-1} \si_{n+1}' = C_n \, \si_n^2. 
\eeq 
The  scaling chosen  in \cite{barnes2022similarity} was to take $C_n=1$ for even $n$, and $C_n=3$ for odd $n$, 
but here we find it convenient to choose the scaling for  $\si_n$ to be such that 
$C_n=\sqrt{3}$  for all $n$. 
However, in several subsequent statements we will leave the choice of scaling arbitrary. 

It follows that, for a given choice of (non-zero) initial tau functions $\si_0$, $\si_1$, if $C_n$ has been fixed for all $n$, then all of the  $\si_n$ for $n\in\Z$ are completely determined     
from (\ref{tauToda}). 

\begin{lemma}\label{ohsiglem} 
The tau functions $\si_n(z)$ that satisfy (\ref{tauToda}) with $C_n=\sqrt{3}$ for all $n$, with initial conditions 
$\sigma_0=\zeta^{-\frac{5}{24}}e^{-\frac{9}{8}\zeta^4}$, 
$\sigma_{1}=3^{1/4}\zeta^{\frac{7}{24}}e^{-\frac{9}{8}\zeta^4-\frac{3}{2}\zeta^2}$, are given by 
\begin{equation}
      \label{sigmarho} \sigma_{n} = 
\begin{cases}
    3^{1/4}\zeta^{\frac{7}{24}}e^{-\frac{9}{8}\zeta^4-\frac{3}{2}n\zeta^2}\rho_{n}(3\zeta^2)& \text{for $n$ odd},\\
    \zeta^{-\frac{5}{24}}e^{-\frac{9}{8}\zeta^4-\frac{3}{2}n\zeta^2}\rho_{n}(3\zeta^2)& \text{for $n$ even},
\end{cases} 
\end{equation} 
where $z=2\ze^3$ and the Ohyama polynomials are evaluated 
at 
$s=3\ze^2$. 
\end{lemma}
\begin{proof}
As in intermediate step, for each choice of parity of $n$, we can substitute the expressions 
(\ref{sigmarho}) into the bilinear 
Toda-type equation \eqref{tauToda} with $C_n=\sqrt{3}$, in order to write  everything in terms of the variable
$\zeta=(z/2)^{\frac{1}{3}}$, and thus obtain
\begin{equation} \label{rhorelation}
      (3\zeta^2+n) \rho_{n}^2-\frac{1}{6\zeta}\left(\frac{\zeta}{2} D^2_{\zeta}\rho_{n}\cdot\rho_{n}+\rho_{n}\rho_{n,\ze} \right) = 
\begin{cases}
    \rho_{n+1}\rho_{n-1} ,& \text{for $n$ odd}\\
    3\zeta^2\rho_{n+1}\rho_{n-1},& \text{for $n$ even}, 
\end{cases} 
\end{equation}
where $D_\ze$ is the Hirota derivative with respect to $\ze$, and the additional subscript denotes an ordinary   derivative: 
$$ 
\rho_{n,\ze} =\frac{\dd}{\dd \ze}\, \rho_n(s).
$$
Note that, for the initial conditions at $n=0,1$, the prescribed choice of $\si_0$ and $\si_1$ implies that 
$\rho_0=1=\rho_1$. 
After making a further change of variables to
the independent variable $s=3\zeta^2$, the relation 
(\ref{rhorelation}) is transformed 
into the recurrence  
(\ref{ohrec}) for Ohyama's polynomials (where there, and throughout, the dot is used to mean the $s$ derivative). 
Since the two sequences $\si_n$ and $\rho_n$ are specified by an equivalent Toda-type relation, modulo the given $n$-dependent rescaling  
and the change of independent variables, and the two initial conditions match, this guarantees that $\si_n$ is specified in terms of the Ohyama polynomials by 
(\ref{sigmarho}) for all $n\in\Z$.    
\end{proof} 

As a consequence of the relation (\ref{taubc}), 
we can reverse the correspondence (\ref{sigmarho}) 
to obtain an analogous relation for $\rho_n(s)$. 
\begin{corollary}
The Ohyama polynomials also obey a relation 
of Burchnall-Chaundy type, namely  
     \begin{equation} \label{newrho}
    \rho_{n+1}\rho_{n-1}+\rho_{n+1}\dot{\rho}_{n-1}-\dot{\rho}_{n+1}\rho_{n-1} = 
\begin{cases}
    s\rho^2_{n}& \text{for $n$ odd,}\\
     \rho^2_{n}& \text{for $n$ even,}
\end{cases} 
\end{equation}
with initial conditions $\rho_0=\rho_{1}=1$, 
where the dot denotes the derivative with respect to $s$.
\end{corollary}

Using \eqref{ohrec}, it was proved in \cite{ohyama2006studies} that $\rho_n(s)$ 
are monic polynomials 
with integer coefficients, which remain invariant up to an overall sign after switching $n\to -n$ and $s\to -s$, with the precise  relation being 
\begin{equation} \label{rhosignsymmetry}
    \rho_{-n}(s) = \begin{cases}
        (-1)^{\frac{n^2}{4}}\,\,\,\rho_n(-s) \qquad\, \text{for even } n, \\
         (-1)^{\frac{n^2-1}{4}}\rho_n(-s) \quad\quad
         \text{for odd } n. 
    \end{cases}
\end{equation}
As a result of the latter property, it would be sufficient to state all subsequent 
results for $n\geq 0$ only, and just use this symmetry to obtain the corresponding statements for negative $n$, 
but for completeness we will usually present everything 
with both choices of sign. 

By using the formulae (\ref{sigmarho}) for even/odd $n$, 
the algebraic solutions of the equation (\ref{P3D7}) are 
manifestly rational functions of the variable $\ze$, being given in terms of the $\rho_n$ by 
\begin{equation} \label{painleverho}
       P_{n} = 
\begin{cases}
    \frac{1}{3\zeta}\rho_{n+1}(3\ze^2)\rho_{n-1}(3\ze^2)/ \rho_{n}(3\ze^2)^2
    & \text{for $n$ odd}\\
    \,\zeta\,\,\rho_{n+1}(3\ze^2)\rho_{n-1}(3\ze^2)/ \rho_{n}(3\ze^2)^2 & \text{for $n$ even}.
\end{cases} 
\end{equation} 
Most of the rest of the paper is devoted to the proof of our main  
result, which can be stated as follows: 
\begin{theorem} \label{OhyamaWronskian} 
For 
$n \in \mathbb{Z}$, the Ohyama polynomials 
$\rho_n(s)$ with $s=3(z/2)^{\frac{2}{3}}$ are given 
in terms of Wronskian determinants,  
by 
\begin{equation} \label{rhowronskian}
    \rho_{n}=
    \begin{cases}
    3^{-\sfrac{1}{4}}\left({s}/{3}\right)^{\frac{|n|-1}{4}}\,\,Wr(W_{\pm1},\dots,W_n),& \text{for $n$ odd}\\
    \left({s}/{3}\right)^{\frac{|n|}{4}}\quad Wr(W_{\pm 1},\dots,W_n),& \text{for  $n$ even}
\end{cases}
\end{equation} 
where the $\pm$ sign is given by $\delta=\sgn(n)$ and 
each Wronskian is taken in the variable $z$. 
The entries $W_n$ are given by the polynomials:
\begin{equation}
   W_{n} = \sum^{2|n|-2}_{k=0}A_{|n|-1,k}\left(\frac{\pm s}{3}\right)^{k}
\end{equation}
where the coefficients $A_{m,k}$ can be non-zero only for $k,m\in\Z_{\geq 0}$ 
and are specified by the recursion 
$$
\frac{k+1}{3}A_{m,k+1}-A_{m,k} +\frac{3\sqrt{3}}{k}A_{m-1,k-2}=0, 
$$ 
together with 
$$A_{m,2m} = \left(\frac{3\sqrt{3}}{2}\right)^m \frac{3^{\frac{1}{4}}}{m!},\quad\text{and}\quad A_{m,0} = 0 \quad \text{for}\quad m\geq1, \quad  A_{0,k}= 0 \quad \text{for}\quad k\geq1 .
$$
\end{theorem}

The preceding result is based on finding the Wronskian  determinants $\theta_n$ which encode the action of   
confluent Darboux transformations applied to the potential $V_0$, given by (\ref{V0}), together with the sequence of 
eigenfunctions, as in (\ref{phitheta}). The  details 
of the proof  are provided in the next section, 
while in section \ref{laxgen} this is followed by a discussion of the 
generating function for the entries of the Wronskian 
and its connection with the Lax pair for (\ref{P3D7}), 
and 
some further technical details about the intertwining operators of 
the Jordan chain are included in an appendix. 

Before proceeding with the complete proof, it is instructive to start by considering the 
first few generalized eigenfunctions  $\psi_i$ 
which appear as entries in the Wronskians  $\theta_{n}$, and how to go about constructing them. After the trivial case $\theta_0=1$, the first case to consider is  
$\theta_1$, which is  a Wronskian of just one function, namely $\psi_1=\phi_0$, an ordinary eigenfunction (with 
eigenvalue zero) for the initial 
Schr\"odinger operator with potential 
$V_0$.   
So we have simply
\begin{equation}
    \theta_1=\psi_1=3^\frac{1}{4}\left(\frac{z}{2}\right)^{\frac{1}{6}}e^{-\frac{3}{2}\left(\frac{z}{2}\right)^{\frac{2}{3}}} = 3^\frac{1}{4}\ze^\frac{1}{2} e^{-\frac{3}{2}\ze^2}, 
\end{equation}
and a solution of the Schr\"odinger equation 
$$ 
\left(\frac{\dd^2}{\dd z^2}+V_0\right) \, \phi_0 =0, \qquad \phi_0 = 3^\frac{1}{4}\ze^\frac{1}{2} e^{-\frac{3}{2}\ze^2}, $$
with $V_0$ given by (\ref{V0}). 
Another independent solution of the same 
Schr\"odinger equation is provided by 
$$ 
\tilde{\phi}_0=(\phi_{-1})^{-1} = 
 3^\frac{1}{4}\ze^\frac{1}{2} e^{\frac{3}{2}\ze^2}, 
 \qquad Wr\big(\phi_{0},(\phi_{-1})^{-1} \big) = \sqrt{3}. 
$$
The new potential obtained by applying a Darboux transformation with the eigenfunction $\phi_0$ is 
$$ 
V_1 = 
V_0 + 2 \frac{\dd^2}{\dd z^2}(\log\phi_0) =
\frac{-7+24\ze^2-36\ze^4}{144\ze^6},   
$$ 
which is the same as the result of applying the BT $T_+$ to find $P_1$ and then using the formula (\ref{VP}). 
By construction, $(\phi_0)^{-1}$ is an eigenfunction of the Schr\"odinger operator with the latter potential, 
but we wish to construct another eigenfunction $\phi_1$, 
such that 
\beq\label{phiwr1}
Wr\big(\phi_{1},(\phi_{0})^{-1} \big) = \sqrt{3}, 
\eeq 
and write it as 
$$ 
\phi_1 = \frac{\theta_2}{\theta_1} 
$$
with $\theta_1=\psi_1$ as before and  $\theta_2$ being a 
Wronskian, that is 
$$ 
\theta_2=\begin{vmatrix}
\psi_1 & \psi_2 \\
\psi_1' & \psi_2'
\end{vmatrix},  
$$
where $\psi_2 $ is to be determined. One way to go about this is to apply the BT $T_+$ once more to obtain the next potential $V_2$ (see Table \ref{FirstExamples}), 
then integrate 
$$ 
V_{2} = V_0 + 2(\log\theta_{2})'' $$ 
twice to find $\theta_2$, and finally use the Wronskian 
formula for $\theta_2$ to extract $\psi_2$ from another integral:  
\begin{equation}\label{psi2sc}
    \frac{\dd}{\dd z}\left( \frac{\psi_2}{\psi_1}\right) = \frac{\theta_2}{\psi_1^2} 
    \implies  \psi_2 
    \propto 
    \left(\frac{z}{2}\right)^{\frac{1}{6}}e^{-\frac{3}{2}\left(\frac{z}{2}\right)^{\frac{2}{3}}}\left( \frac{9}{2}\left(\frac{z}{2}\right)^{\frac{4}{3}}+3\left(\frac{z}{2}\right)^{\frac{2}{3}}\right)
\end{equation}
(where we have ignored an overall choice of scale, 
and the addition of an arbitrary constant multiple of $\psi_1$). 

However, a much more direct approach is to use the Jordan chain associated with the sequence of confluent 
Darboux transformations, upon which  we will base the proof in the next section.
Indeed, the condition (\ref{phiwr1}) is equivalent to the 
Burchnall-Chaundy relation 
$$ 
-\theta_2' = \sqrt{3}\, \theta_1^2
$$
(noting that $\theta_0=1$), and substituting $\theta_1=\psi_1$ as before and $\theta_2$  as a 
$2\times 2 $ Wronskian produces 
$$ 
-\begin{vmatrix}
\psi_1 & \psi_2 \\
\psi_1' & \psi_2'
\end{vmatrix} ' = -\begin{vmatrix}
\psi_1 & \psi_2 \\
\psi_1'' & \psi_2''
\end{vmatrix} =\sqrt{3}\psi_1^2
\implies \psi_1\psi_2''-\psi_2\psi_1''= -\sqrt{3}\psi_1^2, 
$$
so that from $\psi_1''+V_0\psi_1=0 $ we obtain the next step of the Jordan chain, namely the generalized 
eigenfunction equation for $\psi_2$, that is 
$$ 
\psi_2''+V_0\psi_2= -\sqrt{3}\psi_1. 
$$
Then the latter inhomogeneous equation can be solved 
by variation of parameters, using two independent solutions of the homogeneous problem, namely
$\psi_1=\phi_0$ and $\tilde{\psi_1}=(\phi_{-1})^{-1}$, 
leading to 
$$ 
\psi_2 = \psi_1 \left(\int   \tilde{\psi_1}\psi_1\,\dd z\right) -\tilde{\psi_1}\,  \left(\int  \psi_1^2 \,\dd z\right) 
= 3^{-\frac{1}{4}}\zeta^{\frac{1}{2}} 
e^{-\frac{3}{2}\ze^2} \left(\tfrac{9}{2}\ze^4 + 3\ze^2+1+c_1\right). 
$$
Note that the integration constant $c_1$ in the first integral above can be ignored, because adding constant multiples  of $\psi_1$ 
to $\psi_2$ makes no difference to $Wr(\psi_1,\psi_2)$, 
but we can set $c_1=-1$ so that it agrees with (\ref{psi2sc}); 
meanwhile, another integration constant in the second integral has been suppressed, because it corresponds to adding a 
multiple of $\tilde{\psi_1}$ with the wrong type of 
exponential factor $e^{\frac{3}{2}\ze^2}$: we shall return to this point in due course.  
Some other examples of algebraic solutions, potentials, eigenfunctions and Wronskians $\theta_n $ for small $n$ are listed in Table  \ref{FirstExamples}.


\begin{table}
    
    \hspace*{-1.625cm}
    \renewcommand{\arraystretch}{3}%
    \begin{tabular}{ |c|c|c|c|c| } 
    
\hline
$n = \be/2$ & $0$ & $1$ & $2$ & $3$ \\
\hline
$P_{n}$ & $\zeta$ & $\frac{3\zeta^2+1}{3\zeta}$ & $ \frac{\zeta\left(9\zeta^4+12\zeta^2+5\right)}{\left( 3\zeta^2+1\right)^2}$ & $\frac{(3\zeta^2+1)(81\zeta^8+270\zeta^6+360\zeta^4+210\zeta^2+35)}{3\zeta(9\zeta^4+12\zeta^2+5)^2}$\\ 
$V_{n}$ &$-\frac{36\zeta^4-5}{144\zeta^6}$& $-\frac{36\zeta^4-24\zeta^2+7}{144\zeta^6}$ &  $-\frac{324\zeta^8-216\zeta^6+135\zeta^4-30\zeta^2-5}{144\zeta^6(3\zeta^2+1)^2}$ & $-\frac{2916\zeta^{12}+1944\zeta^{10}+1215\zeta^8+1080\zeta^6-630\zeta^4+175}{144\zeta^6(9\zeta^4+12\zeta^2+5)^2}$ \\ 
$\sigma_{n}$ & $\zeta^{-\frac{5}{24}}e^{-\frac{9}{8}\zeta^4}$ & $3^{1/4}\zeta^{\frac{7}{24}}e^{-\frac{9}{8}\zeta^4-\frac{3}{2}\zeta^2}$ & $(3\zeta^2+1)\zeta^{-\frac{5}{24}}e^{-\frac{9}{8}\zeta^4-3\zeta^2}$ & $3^{1/4}(9\zeta^4+12\zeta^2+5)\zeta^{\frac{7}{24}}e^{-\frac{9}{8}\zeta^4-\frac{9}{2}\zeta^2}$\\
$\theta_{n}$ & $1$ &  $3^{1/4}\zeta^{\frac{1}{2}}e^{-\frac{3}{2}\zeta^2}$ & $(3\zeta^2+1)e^{-3\zeta^2}$ & $3^{1/4}(9\zeta^4+12\zeta^2+5)\zeta^{\frac{1}{2}}e^{-\frac{9}{2}\zeta^2}$ \\
$\phi_{n}$ & $3^{1/4}\zeta^{\frac{1}{2}}e^{-\frac{3}{2}\zeta^2}$ & $3^{-1/4}(3\zeta^2+1)\zeta^{-\frac{1}{2}}e^{-\frac{3}{2}\zeta^2}$ & $3^{1/4}\frac{(9\zeta^4+12\zeta^2+5)}{(3\zeta^2+1)}\zeta^{\frac{1}{2}}e^{-\frac{3}{2}\zeta^2}$ & $3^{-1/4}\frac{(81\zeta^8+270\zeta^6+360\zeta^4+210\zeta^2+35)}{(9\zeta^4+12\zeta^2+5)}\zeta^{-\frac{1}{2}}e^{-\frac{3}{2}\zeta^2}$\\
$\rho_n$ & 1 & 1 & $s+1$ & $s^2+4s+5$ \\ 
\hline
\end{tabular}
    \caption{Algebraic solutions of \eqref{P3D7}, and associated potentials, tau functions, Wronskians,   eigenfunctions 
    in terms of $\zeta = \left(\frac{z}{2}\right)^{\frac{1}{3}}$, 
    with Ohyama polynomials in $s=3\ze^2$ for $n = 0,1,2,3$. }
    \label{FirstExamples}
\end{table}

\section{Proof of main theorem} 
\label{OhyamaSection}

In this section, we present the proof of Theorem \ref{OhyamaWronskian}, 
based on the properties 
of confluent Darboux transformations and the 
associated 
generalized eigenfunctions, forming a Jordan chain,   
in the context of the 
sequence of 
potentials $V_n$ associated with the 
algebraic solutions of 
the Painlev\'e III  ($D_7$) equation.  
Since we are considering the effect of the BTs (\ref{Backlund}) applied successively to the 
algebraic seed solution $P_0=\ze$ for $\be =0$, and the corresponding sequence of solutions 
$P_n$ with parameter $\be=2n$ for $n\in\Z$, it will be convenient to use $\psi_n$ 
to denote each member of an 
associated set of generalized eigenfunctions, labelled by the same integer $n$.   
Thus, after $|n|$ applications of a confluent Darboux tranformation applied to the 
Schr\"odinger operator with potential $V_0$ given by (\ref{V0}), we find two different sequences  
of generalized eigenfunctions 
namely $\psi_{\delta},\dots,\psi_{n}$, where $\delta=\pm 1=\sgn(n)$   
(so one for positive values of $n$, the other for negative $n$). 
This allows us to write 
Wronskian formulae for the eigenfunctions (with eigenvalue zero) 
of each of the Schr\"odinger operators with potential $V_n$, which  are valid for 
all $n\in\Z$. So we have 
$$ \phi_n =  \frac{\theta_{n+1}}{\theta_{n}}, \qquad \text{with} \quad 
\theta_{n}=Wr(\psi_\delta,\dots,\psi_{n}), 
$$
according to the choice of sign $\delta$, 
where 
$$ 
\left(\frac{\dd^2}{\dd z^2}+V_n\right)\, \phi_n = 0 
$$ 
holds for each integer $n$. 

The main idea of the proof is to use the Frobenius method to obtain an explicit recursion for the generalized eigenfunctions $\psi_n$, which (up to an exponential factor and a prefactor of $\sqrt{\ze}$) 
are given by polynomials in $\ze^2$. However, before we proceed with this it is convenient to state a 
slightly technical result 
on the coefficients that will appear in the Frobenius expansion, which reflects the fact that the whole construction has a certain symmetry property under exchanging $n\leftrightarrow -n$. (This can also be seen from the property (\ref{rhosignsymmetry}) 
of the Ohyama polynomials.) 
\begin{lemma}  \label{ACoefficientSymmetry}
Let coefficients $A_{m,k}$, labelled by integers $m,k$ with $k\geq 0$, be generated  recursively by the relations \beq\label{deltaArelation}
\begin{array}{rcl}
{A}_{m,k}&=& 
(3a/k) 
{A}_{m-\delta,k-2}+\frac{k+1}{3}\delta {A}_{m,k+1} \quad \text{for}\quad m\neq 0,\,\, k\geq 2, \\
{A}_{m,1}&=&\frac{2}{3}\delta{A}_{m,2}
\end{array} 
\eeq 
 with boundary conditions  
\beq\label{Abcs} 
{A}_{m,0}= 0 \quad \text{for}\quad m\neq 0, 
\qquad 
{A}_{m,2|m|}=\left(\frac{3a}{2} \right)^{|m|} \frac{c}{|m|!}, \quad \text{for all}\quad m, 
\eeq 
where $\delta=\sgn(n)$ and $a,c\in\C$ are arbitrary non-zero 
parameters.  
Then the coefficients satisfy 
\begin{equation} 
    A_{m,k}=(-1)^{k}A_{-m,k}  \label{Acoefficientsymmetry}
\end{equation}
for all $m\in\Z$.
\end{lemma}
\begin{proof} The result follows by a straightforward induction on $|m|$. First of all, the result is trivially true for $m=0$, with $A_{0,0}=c$. Now we wish to show by 
induction that, for each $m\neq 0$, the recursion relations (\ref{deltaArelation}) and boundary values \eqref{Abcs} completely determine the coefficients $A_{m,k}$ for all $k\geq 0$, and moreover that 
$A_{m,k}=0$ when $k>2|m|$. When $m>0$, so $\delta=1$, it is convenient to start by taking $k=2m$ in  
(\ref{deltaArelation}), in which case the boundary values
$A_{m,2m}$ and $A_{m-1,2m-2}$ give 
$$ 
\left(\frac{3a}{2} \right)^{m} \frac{c}{m!} 
=\left(\frac{3a}{2m}\right)\, \left(\frac{3a}{2} \right)^{m-1} \frac{c}{(m-1)!} 
+\frac{2m+1}{3}\, A_{m,2m+1}
 $$
 which  
implies that $A_{m,2m+1}=0$, and thus, by increasing the value of $k$ and using the inductive hypothesis, we see 
that $A_{m,k}=0$ for all $k\geq 2m+1$, as required. 
Hence, by decreasing $k$ at each step, we can use the recursion (\ref{deltaArelation}) to determine the other (non-zero) coefficients $A_{m,k}$ 
for $1\leq k\leq 2m-1$, while $A_{m,0}=0$ is fixed by \eqref{Abcs}. 
The coefficients $A_{m,k}$ with $m<0$ are determined  similarly by the recursion with $\delta =-1$.  
Then the statement is trivially true for 
$k>2|m|$ and $k=0$, and for other values of $k$ it follows 
from the fact that \eqref{deltaArelation} and the 
boundary conditions are left invariant under replacing 
$A_{m,k}\to(-1)^{k}A_{-m,k}$, $\delta\to-\delta$. 
\end{proof} 
We are now ready to present generalized eigenfunctions $\psi_n$ for the  potential $V_0$ given by \eqref{V0}, associated with the algebraic seed solution of the Painlev\'e III ($D_7$) equation, which is given in terms of $\ze$ by $P_0=\zeta$. 
\begin{proposition} \label{psiproposition}
    For $n\in\mathbb{Z}\setminus\{ 0\}$, 
    $\psi_{n}$ defined by
\begin{equation}\label{psinsol}
    \psi_{n} = \zeta^{\frac{1}{2}}e^{-\frac{3}{2}\delta\zeta^2}\sum^{2|n|-2}_{k=0}A_{|n|-1,k}(\delta\zeta^{2})^{k}
\end{equation}
where $\delta=\sgn(n)=\pm 1$, $a,c\in\mathbb{C}$ are arbitrary non-zero constants, 
and the coefficients $A_{m,k}$ are subject to the relations 
\eqref{deltaArelation} and \eqref{Abcs}, 
are generalized eigenfunctions for the 
potential $V_0$, 
satisfying 
\begin{align} \label{JordanChainpsi}
\begin{split}
    {\psi}_{n}''+V_0{\psi}_{n}&=-\delta a{\psi}_{n-\delta},
\end{split}
\end{align}
where $'$ denotes the derivative with respect to  $z$, and $\psi_0=0$. 
\end{proposition}
\begin{proof}
Firstly considering the case 
$n>0$, changing variables from $z$ to $\zeta$ in 
\eqref{JordanChainpsi}, writing 
$$\psi_n (z)=\hat{\psi}_n(\ze)
$$ 
and substituting in the formula \eqref{V0} for $V_0$ in terms of $\zeta$ 
gives
\begin{equation} \label{SubzetaV0}
    4\zeta^2\hat{\psi}_{n}''-8\zeta\hat{\psi}_{n}'-36\zeta^4\hat{\psi}_{n}+5\hat{\psi}_{n}=-144a\zeta^6\hat{\psi}_{n-1} , 
\end{equation}
where the primes on $\hat{\psi}_n$ denote derivatives with respect to the argument $\ze$.
Define $w_n = w_n(\ze)$ for $n\geq 0$ by 
\begin{equation} \label{psiw}
    \hat{\psi}_{n}(\ze) =\zeta^{\frac{1}{2}}e^{-\frac{3}{2}\zeta^2}w_{n-1}(\zeta), 
\end{equation}
where shifting down the index to $w_{n-1}$ makes certain powers of $\zeta$ easier to keep track of later, 
and we set $w_{-1}=0$ to be consistent with the convention that 
$\psi_0=0$. 
This simplifies the modified form of the Jordan chain to be 
\beq
    \zeta^2\frac{\rd^2 w_{n}}{\rd\ze^2}-\left(6\zeta^3+\zeta \right)\frac{\rd w_{n}}{\rd\ze}+36a\zeta^6w_{n-1}=0, \label{wrecursion}
\eeq
Taking the initial constant solution $w_0=c\neq 0$ when $n=0$ in \eqref{wrecursion}, we may follow Frobenius' method by taking a series solution   
\begin{equation}
    w_{n}=\sum^{\infty}_{k=0}A_{n,k}\zeta^k,
\end{equation}
which, when  substituted into the preceding relation between $w_n$ and $w_{n-1}$, 
produces 
\begin{equation} \label{Frobenius}
    \sum^{\infty}_{k=0}k(k-2)A_{n,k}\zeta^{k-2} - 6\sum^{\infty}_{k=0}kA_{n,k}\zeta^k + 36a\sum^{\infty}_{k=0}A_{n-1,k}\zeta^{k+4} = 0.
\end{equation}
Shifting $k\rightarrow k-2$ and $k\rightarrow k-6$, respectively, and collecting terms with the same power of $\ze$ in \eqref{Frobenius}, 
we begin by 
explicitly calculating the coefficients for the 
first 6 values of $k$ in the sum. 
As each coefficient must vanish, we  see that  $A_{n,0}$ and $A_{n,2}$ can be freely chosen, while $A_{n,1}$, $A_{n,3}$ and $A_{n,5}$ must all be zero, and for $k=5$ we have that $A_{n,4} = \frac{3}{2}A_{n,2}$. 
The remaining infinite sum gives the relations defining $A_{n,k}$ 
for $k\geq 6$, which are  
obtained as 
$$
    A_{n,k}=\frac{6}{k}A_{n,k-2}-\frac{36a}{k(k-2)}A_{n-1,k-6}.
$$
From this and $A_{n,1}=A_{n,3}=A_{n,5}=0$ we can see that for all odd $k$, $A_{n,k}=0$. We also choose to take  $A_{n,0}=0$ for $n \neq 0$, since including a  non-zero $A_{n,0}$ is equivalent to adding a multiple of the homogeneous solution $\psi_1$. (In the context of the 
 Wronskian form of the solution with entries 
$\psi_{n}$, this amounts to adding one column onto another, having no effect on the determinant.) \par
From the examples of  $\psi_1$ and $\psi_2$ calculated previously, we would expect there to be some choice of $A_{n,2}$ which truncates this infinite sum, which follows from \eqref{wrecursion} if we assume that the $w_{n}$ are polynomials in $\ze$. If  $w_{n}$ and $w_{n-1}$ are polynomials of degree $m$ and $p$ respectively, then in 
the relation between them, the highest powers appearing are $m+2$ and $p+6$; hence in order for them to agree, we require $m=p+4$. 
Thus if the $w_{n}$  are a sequence of polynomials, each of  
degree $4$ greater than the previous one, then the fact that $w_0=c$ has degree $0$ implies that $w_{n}$ must be of degree $4n$, and so we must have $A_{n,k}=0$ for $k>4n$. All greater odd $k$ terms are zero automatically, so the first that must be set to zero is 
$$
    0=A_{n,4n+2}= \frac{6}{4n+2}A_{n,4n} - \frac{36a}{4n(4n+2)}A_{n-1,4n-4} \implies A_{n,4n}=\frac{3a}{2n}A_{n-1,4n-4}.
$$ 
By induction, setting $A_{n,4n+2}=0$ also sets  $A_{n,k}=0$ for 
all higher $k$  as well, as they only rely linearly on  it and on values of $A_{n-1,k}$  beyond the point at which all these coefficients are fixed to zero. Since $A_{0,0}=c=w_0$,  iterating the above recursion for $A_{n,4n}$ in terms of $A_{n-1,4n-4}$ 
gives 
    $$
    A_{n,4n}=\left( \frac{3a}{2} \right)^n \frac{c}{n!}.
    $$
\par 
The relation for $A_{n,k}$ can also be rearranged to step down in $k$ instead of up, in the form 
$$
    A_{n,k}=\frac{k+2}{6}A_{n,k+2}+\frac{6a}{k}A_{n-1,k-4},
$$
and so, for a given $n$, starting from the above formula for $A_{n,4n}$, this relation used to find all terms $A_{n,k}$ with lower $k$ all the way down to $k=2$, with the inductive assumption that the 
previous terms  $A_{n-1,k}$ have been found already. This produces a sequence of polynomials $w_n=\sum^{n}_{k=0}A_{n,k}\zeta^{2k}$ for $n\geq 0$ 
(with  no odd powers of $\zeta$ in the sum), whose coefficients are 
determined by 
\eqref{deltaArelation} and \eqref{Abcs}, and as a consequence of 
\eqref{Acoefficientsymmetry} we also have 
$$ 
A_{0,k}= 0 \qquad \text{for}\quad k\geq1.
$$ 
This completes the proof of the result on the 
generalized eigenfunctions $\psi_{n}$ for $n>0$, 
corresponding to the statement with $\delta=1$. 

For the case of negative $n$, we instead let
\begin{equation} \label{barpsiw}
    \psi_n(z)=\hat{\psi}_{n}(\ze)=\zeta^{\frac{1}{2}}e^{\frac{3}{2}\zeta^2}{u}_{n+1}(\zeta), 
\end{equation}
and then substituting into the Jordan chain 
equation \eqref{SubzetaV0} with the replacement $a\to -a$ (corresponding to $\delta=-1$)  
results in a relation for ${u}_{n}$ which is subtly different from 
\eqref{wrecursion}, namely 
$$ 
    \zeta^2\frac{\rd^2 u_{n}}{\rd\ze^2}+\left(6\zeta^3-\zeta \right)\frac{\rd u_{n}}{\rd\ze}-36a\zeta^6{u}_{n+1}=0, 
$$
where we set $u_{1}=0$ to ensure the validity of the above when $n=0$. 
Mutatis mutandis, using the symmetry properties of the coefficients $A_{n,k}$ as in Lemma \ref{ACoefficientSymmetry}, the rest of the proof proceeds by induction in the same way as for the case of positive $n$. 
Hence we obtain the expression \eqref{psinsol} for the sequence of generalized eigenfunctions of the Schr\"odinger operator with 
potential $V_0$, valid for both choices of the sign $\delta=\sgn (n)=\pm 1$.
\end{proof}

The application of Frobenius' method in the above proof shows that the 
factors $w_n$ in \eqref{psiw} and $u_n$ in \eqref{barpsiw} are given by infinite series in general, as different choices of the coefficients $A_{n,0}$ and $A_{n,1}$ are allowed at each step, corresponding to the freedom to add arbitrary multiples of $\psi_k$ with $|k|<|n|$ to each 
generalized eigenfunction $\psi_n$. While such choices lead to 
valid Darboux transformations of the potential $V_0$, 
they do not result in the correct forms for the Wronskians 
$\theta_n$ to produce the Ohyama polynomials, 
as required by Theorem \ref{OhyamaWronskian}, whose proof 
will be presented shortly.

\begin{remark} It is worth commenting on the choice of scaling at this stage, as reflected in the parameters $a,c$. The constant $c$ 
determines the choice of scaling for $\psi_1$, the initial eigenfunction for the Schr\"odinger operator with potential $V_0$, while the parameter $a$ 
corresponds to the Wronskian between pairs of eigenfunctions generated by subsequent iterations of the confluent Darboux transformation (see the Appendix for more details). In order to be consistent with 
\cite{sato}, we usually fix $a=\sqrt{3}$ and $c=3^{\frac{1}{4}}$, e.g. 
in the statement of Theorem \ref{OhyamaWronskian}. 
However, other choices make the form of the coefficients 
 $A_{n,k}$ and the relations between them somewhat simpler. 
 For instance, the choice $c=1$ and $a=\frac{2}{3}$ can be taken, which gives 
 \begin{align}
 A_{n,2n}&=\frac{1}{n!},\nonumber \\ 
 A_{n,2n-1} &= \frac{2}{9}(2n+1)\frac{1}{(n-1)!},\nonumber \\
    A_{n,2n-2}&= \frac{2}{81}(4n^2+10n+9)\frac{1}{(n-2)!},\nonumber \\
    A_{n,2n-3}&= \frac{4}{2187}(2n-1)(4n^3+18n^2+23n-15)\frac{1}{(n-2)!},\nonumber \\
    A_{n,2n-4}&= \frac{2}{19683}(16n^5+128n^4+356n^3+292n^2+513n+1125)\frac{1}{(n-3)!},\nonumber
\end{align}
and so on, which appear to be the most concise explicit forms of these 
coefficients. 
It is also possible to produce an exact generating function $\Psi$ for 
the generalized eigenfunctions that appear as entries in the Wronskians: in the case 
$\delta=+1$, upon multiplying \eqref{JordanChainpsi}  by $\la^{n-1}$ and summing from $n=1$ to $\infty$, we find 
$$
\Psi''+V_0\Psi = -a\la\Psi, 
$$
which is a Schr\"odinger equation with a non-zero eigenvalue term. 
After fixing the scale so that $a=-1$, an explicit 
formula for $\Psi$ is given in the next section (see Proposition \ref{genprop}). 
\end{remark}

We now complete the proof of our main result 
on Ohyama polynomials.  

\begin{proof}[Proof of Theorem \ref{OhyamaWronskian}]
Recall from the results of Lemma \ref{dtlem}, Proposition \ref{recprop}, and  Lemma \ref{ohsiglem}, that the Ohyama polynomials 
$\rho_n$ are equivalent (up to some scale factors) to a sequence of 
tau functions $\si_n$ for the Painlev\'e III ($D_7$) equation. In turn, via the formula 
$$
V_n = 2 \frac{\rd^2}{\rd z^2} \log\si_n, \qquad n\in\Z, 
$$ 
these tau functions correspond to a sequence of Schr\"odinger potentials obtained from successive application of confluent 
Darboux transformations applied to the initial potential $V_0$, with 
 one Jordan chain for $n\geq  0$ and another for $n\leq 0$. 
General results on confluent Darboux transformations (which are described in more detail in the Appendix) 
imply that the tau functions and new eigenfunctions obtained 
via this process can be written in terms of quantities $\theta_n$, as in \eqref{thetasigma} and \eqref{phitheta}, where 
$$\theta_{n}= Wr(\psi_{\delta},\dots,\psi_{n})$$
are 
Wronskians of generalized eigenfunctions $\psi_j$. Moreover, the 
Burchnall-Chaundy relation \eqref{sigbc} fixes a choice of 
normalization for the new eigenfunction introduced at each stage. 

The problem is then to specify the explicit choices of the 
generalized eigenfunctions  
that appear as entries in the Wronskians $\theta_n$. The proof of Proposition \ref{psiproposition} 
shows that in general, these have the form 
$$\psi_{n}=\zeta^{\frac{1}{2}}e^{-\frac{3}{2}\delta\zeta^2}W_n(\zeta), $$
(with $W_n=w_{n-1}$ for $n>0$,  $W_n=u_{n+1}$ for $n<0$, respectively), 
where each $W_n$ is given by an infinite series in $\ze$, unless 
suitable coefficients are fixed to be zero, in which case it is the 
polynomial
\beq\label{Wpoly}
W_n = \sum^{2|n|-2}_{k=0}A_{|n|-1,k}(\delta\zeta^{2})^{k}. 
\eeq
Using the Wronskian identity 
$$Wr(g\,f_1,g\, f_2,\dots,g\, f_n)=g^n\, Wr(f_1,f_2,\dots,f_n), $$
we can extract all the prefactors $g=\zeta^{\frac{1}{2}}e^{-\frac{3}{2}\delta\zeta^2}$ to obtain the formula 
\beq\label{thetafin}
     \theta_{n}= Wr(\psi_{\delta},\dots,\psi_{n}) = \zeta^{\frac{|n|}{2}}e^{-\frac{3}{2}n\zeta^2}Wr(W_{\delta},\dots,W_n).
\eeq 
On the other hand,  using \eqref{thetasigma} and \eqref{sigmarho} to 
express $\theta_{n}$ in terms of $\rho_{n}$, we find 
\begin{equation} \label{thetarho}
       \theta_{n} = 
\begin{cases}
    3^{1/4}\zeta^\frac{1}{2}e^{-\frac{3}{2}n\zeta^2}\rho_{n}(3\zeta^2) & \text{for $n$ odd,}\\
    e^{-\frac{3}{2}n\zeta^2}\rho_{n}(3\zeta^2)& \text{for $n$ even.}
\end{cases} 
\end{equation}
By comparing \eqref{thetafin} with \eqref{thetarho} we obtain, 
\begin{equation}\label{rhozeta}
    \rho_{n}=
    \begin{cases}
    3^{-\frac{1}{4}}\zeta^{\frac{|n|-1}{2}}Wr(W_{\delta},\dots,W_n)& \text{for $n$ odd},\\
    \zeta^{\frac{|n|}{2}}Wr(W_{\delta},\dots,W_n)& \text{for $n$ even},
\end{cases}
\end{equation}
where the Wronskian entries are as in \eqref{Wpoly}. 
Converting (\ref{rhozeta}) to the variable $s=3\zeta^2$ gives the statement of the theorem.
\end{proof} 

\begin{remark}
The formula \eqref{rhozeta} does not make all the properties of 
the Ohyama polynomials obvious. In particular, the presence of the derivatives $\frac{d}{dz}=\frac{1}{6\zeta^2}\frac{d}{d\zeta}$ in the 
Wronskian means that it can be inferred immediately that 
\begin{equation} \label{rationalRho}
    \rho_{n}=\zeta^{L_n}D_{n}(\zeta),
\end{equation} 
for some $L_n\in\Z$, where $D_n$ is a polynomial in $\ze$ with 
non-zero constant term, and it is also easy to see that 
$\rho_n$ is an even function of $\ze$; but it is not clear why 
$L_n\geq 0$, so the best conclusion 
that can be made from \eqref{rationalRho} is that 
$\rho_n$ is a Laurent polynomial in $s=3\zeta^2$. Furthermore, 
it is not immediately apparent why $\rho_n(s)$ should have integer 
coefficients, which is another feature of the Ohyama polynomials, 
with the normalization chosen as in \cite{sato}.
\end{remark}

\section{Lax pair for Painlev\'e III ($D7$) and generating functions}
\label{laxgen}

In this section, we consider  a Lax pair for the 
Painlev\'e
III ($D_7$) equation, and present a particular solution to this linear system. This solution is then used to derive a generating function for the generalized eigenfunctions $\psi_i$ which appear as Wronskian entries in the tau functions for the algebraic solutions. 

A scalar Lax pair for the Painlev\'e III ($D_7$) equation 
(\ref{P3D7}) is given by 
\begin{equation} \label{PIIILax}
    \begin{cases}
        \Psi_{zz} +V\Psi = \lambda\Psi, \\
        \Psi_\lambda = \frac{1}{2}\left(\frac{z}{\lambda} - \frac{P}{\lambda^2} \right)\Psi_z - \frac{1}{4}\left(\frac{1}{\lambda} - \frac{P'} {\lambda^2}\right)\Psi, 
    \end{cases}
\end{equation} 
where $V$ is defined in terms of $P$ and $P'$ (its 
$z$ derivative) by (\ref{VP}). With $V$ specified in this way, the compatibility condition $\Psi_{zz\la}=\Psi_{\la zz}$ produces (\ref{P3D7}) directly. Alternatively, one can start with $V,P$ as unspecified functions, and 
then the compatibility condition yields the system 
$$ 
\begin{array}{rcl}
P'''+4VP'+2V'P & = & 0, \\ 
P'+\tfrac{z}{2}V' +V & = & 0, 
\end{array}
$$ 
from which both (\ref{VP}) and (\ref{P3D7}) can be derived, with the parameter $\be$ appearing as an 
integration constant. 
This Lax pair can be derived by 
applying the similarity reduction (\ref{hodo}), as 
found in \cite{hone1999associated} and  used in 
\cite{barnes2022similarity}, to the Lax pair for 
the Camassa-Holm equation, 
which is related via a reciprocal 
(hodograph-type) 
transformation to the Lax pair for the first negative KdV flow \cite{hone2002prolongation}, that is 
\begin{equation} \label{reciprocalLax}
    \begin{cases}
        \Phi_{XX}+\Bar{V}\Phi = \xi\Phi, \\
        \Phi_T = A\Phi_X - \frac{1}{2}A_X\Phi,
    \end{cases}
\end{equation}
where  $\xi$ is the  spectral parameter and  
$$ 
    \Bar{V} = - \frac{p_{XX}}{2p} + \frac{p_X^2}{4p^2} - \frac{1}{4p^2}, \qquad A=\frac{p}{2\xi} .
$$ 
We can simultaneously perform a scaling similarity reduction on the coefficients and the wave function of this Lax pair, 
by setting 
$$   
 p(X,T)=T^{-\sfrac{1}{2}}P(z), \qquad z = XT^{\sfrac{1}{2}}, 
    \qquad
    \Phi(X,T;\xi) = \la^{\sfrac{1}{4}}{\Psi}(z,\lambda), \qquad\lambda = \xi T^{-1},
$$ 
which 
transforms the linear system  
 (\ref{reciprocalLax}) into  
the Lax pair \eqref{PIIILax} above.

Recently, Buckingham \& Miller \cite{buckingham2022algebraic} presented  a Riemann-Hilbert representation of the algebraic solutions of the Painlev\'e III ($D_7$) equation, which involved solving the linear system coming from an alternative Lax pair of Jimbo-Miwa type in terms of Airy functions. Since 
 \eqref{PIIILax} is connected to the Jimbo-Miwa Lax pair by a gauge transformation, we can follow their lead  
 somewhat 
 and see that the wave function $\Psi$ can be solved in terms of Airy functions when we take the seed solution $P_0=(z/2)^{\sfrac{1}{3}}$. 
 \begin{proposition}
    The Lax pair for the algebraic seed solution of 
    Painlev\'e III ($D_7$) with $\be=0$,  given by the linear system \eqref{PIIILax},  has the general solution 
    \begin{equation}\label{Psigen}
    \Psi = \lambda^{-\frac{1}{6}}\zeta^{\frac{1}{2}}\left[c_1\Ai\left(\left(\frac{9}{\lambda^2}\right)^{\frac{1}{3}}\left(\frac{1}{4}+\lambda\zeta^2 \right)\right) + c_2\Bi\left(\left(\frac{9}{\lambda^2}\right)^{\frac{1}{3}}\left(\frac{1}{4}+\lambda\zeta^2 \right)\right) \right],
\end{equation}
for arbitrary constants $c_1, c_2$.
\end{proposition}
\begin{proof}
For the algebraic solution of Painlev\'e III ($D_7$),  we substitute $V_0$ from (\ref{V0}) and $P_0=(z/2)^{\sfrac{1}{3}}=\ze$ into the Lax pair. Then, 
by rewriting the first equation in (\ref{PIIILax}) in terms of $\zeta$, 
we obtain
$$
    4\zeta^2\Psi_{\zeta\zeta}-8\zeta\Psi_\zeta - (36\zeta^4-5)\Psi = 144\lambda\zeta^6\Psi,
$$
and making a further change of variables to $w=\frac{1}{4}+\lambda\zeta^2$, we  get
$$
    16\left(w-\frac{1}{4}\right)^2\Psi_{ww}-8\left(w-\frac{1}{4}\right)\Psi_w=
    \Big(144\big(w-\tfrac{1}{4}\big) 
    +36\la^{-2}\big(w-\tfrac{1}{4}\big)^2  
    -5\Big)\Psi.
$$ 
To clean up the right-hand side of the preceding equation, 
we set 
$$
   \Psi = \left(w-\frac{1}{4}\right)^{\frac{1}{4}}\vartheta(w)=\lambda^{\frac{1}{4}}\zeta^{\frac{1}{2}}\vartheta\left(\frac{1}{4}+\lambda\zeta^2\right),
$$
which gives a scaled Airy equation for $\vartheta(w)$, that is 
$$
    \vartheta_{ww} = \left(\frac{9}{\lambda^2}w \right) \vartheta, 
$$
with general solution
$$ 
    \vartheta(w)= \tilde{A}\Ai\left(\left(\frac{9}{\lambda^2}\right)^{\frac{1}{3}}w\right) + \tilde{B}\Bi\left(\left(\frac{9}{\lambda^2}\right)^{\frac{1}{3}}w\right)
$$ 
where $\Ai(w), \Bi(w)$ are the standard Airy functions and $\tilde{A}, \tilde{B}$ are independent of $w$ but otherwise arbitrary. 
Hence the solution of the first part of the Lax pair is given by
\begin{equation} \label{FirstLaxSolution}
    \Psi = \lambda^{\frac{1}{4}}\zeta^{\frac{1}{2}}\left[\tilde{A}(\lambda)\Ai\left(\left(\frac{9}{\lambda^2}\right)^{\frac{1}{3}}\left(\frac{1}{4}+\lambda\zeta^2 \right)\right) + \tilde{B}(\lambda)\Bi\left(\left(\frac{9}{\lambda^2}\right)^{\frac{1}{3}}\left(\frac{1}{4}+\lambda\zeta^2 \right)\right) \right],
\end{equation} 
where 
$A(\lambda),B(\lambda)$ are arbitrary functions of $\la$ (that do not depend on $\zeta$).  
We can fix these coefficients using terms via the second part of the Lax pair to get a general solution.  
Substituting the solution (\ref{FirstLaxSolution}) 
into the second equation 
in (\ref{PIIILax}) gives an equation in terms of $\Ai, \Bi$ and their derivatives $\Ai',\Bi'$. The derivative terms 
cancel identically, hence place no restriction on the coefficients $A(\lambda), B(\lambda)$. The remaining part can be written concisely as 
$$
        \lambda^{\frac{1}{4}}\zeta^{\frac{1}{2}}\left(A'(\lambda)\Ai\left(\left(\frac{9}{\lambda^2}\right)^{\frac{1}{3}}\left(\frac{1}{4}+\lambda\zeta^2 \right)\right)  + B'(\lambda)\Bi\left(\left(\frac{9}{\lambda^2}\right)^{\frac{1}{3}}\left(\frac{1}{4}+\lambda\zeta^2 \right)\right)\right) = -\frac{5}{12}\lambda^{-1}\Psi,
$$
which yields separate equations relating the coefficients in front of the independent functions $\Ai$ and $\Bi$ appearing on each side. 
For the coefficient of $\Ai$ we obtain
$$
    A'(\lambda)=-\frac{5}{12\lambda}A(\lambda) 
    \implies 
     A(\lambda) = c_1\lambda^{-\frac{5}{12}},
$$
and  likewise for the coefficient of $\Bi$ we get the same linear ODE, so that $B(\lambda) = c_2\lambda^{-\frac{5}{12}}$, for arbitrary constants $c_1, c_2$. Hence overall we have the solution 
(\ref{Psigen}),  in terms of the  standard Airy functions $\Ai,\Bi$, as required. 
\end{proof}
Although the independent variable in (\ref{P3D7}) is $z$, it is more convenient for the algebraic solutions to use 
$\zeta$ instead, so we have written $\Psi(\ze,\lambda)$ 
in terms of the latter variable in (\ref{Psigen}).  
Returning to  the consideration of the generalized eigenfunctions $\psi_i$ which are entries of the Wronskians $\theta_n$, we noted previously that a generating function $\Psi$ 
for these entries would be a solution of the 
Schr\"odinger equation  
\begin{equation}\label{gensch}
   \left( \frac{\dd^2}{\dd z^2} +V_0\right)\Psi = \lambda\Psi,
\end{equation}
with $V_0$ given by (\ref{V0}), 
which is exactly the $z$ part of the Lax pair 
for Painlev\'e III ($D_7$), so from the proof of the preceding proposition, we see that such a $\Psi$ must be of the form \eqref{FirstLaxSolution}, leading us to the 
following result.
\begin{proposition}\label{genprop}
For the generalized eigenfunctions 
$\psi_i$ with $i>0$ of the algebraic potential $V_0$ given by (\ref{V0}), 
the generating function 
\beq\label{posgenfn} 
\Psi = \sum^{\infty}_{n=1}\psi_{n}\lambda^{n-1}
\eeq 
is given by 
\begin{equation} \label{PositivePsiGen}
    \Psi = 3^{\frac{5}{12}}\lambda^{-\frac{1}{6}}e^{\frac{1}{4\lambda}}\sqrt{2\pi\zeta}\,\Ai\left(\left(\frac{9}{\lambda^2}\right)^{\frac{1}{3}}\left(\frac{1}{4}+\lambda\zeta^2 \right)\right),
\end{equation}
while the generating function 
$$
\tilde{\Psi} = \sum^{\infty}_{n=1}\psi_{-n}\lambda^{n-1}
$$ 
for the generalized eigenfunctions 
$\psi_i$ with $i<0$
is given by 
\begin{equation} \label{NegativePsiGen}
    \tilde{\Psi} = 3^{\frac{5}{12}}\lambda^{-\frac{1}{6}}e^{-\frac{1}{4\lambda}}
    \sqrt{\frac{\pi\zeta}{2}}
    \, \Bi\left(\left(\frac{9}{\lambda^2}\right)^{\frac{1}{3}}\left(\frac{1}{4}+\lambda\zeta^2 \right)\right).
\end{equation}
\end{proposition}
\begin{proof} 
Since (\ref{posgenfn}) is a solution of the 
Schr\"odinger equation (\ref{gensch}), we need to  
choose the coefficients $\tilde{A}(\la),\tilde{B}(\la)$ 
in 
\eqref{FirstLaxSolution} so that we have a solution with the correct asymptotic behaviour as $\la\to 0$, which means 
considering the asymptotics of the Airy 
functions as their argument goes to infinity. 
From the results in section $\S$9.7(ii) 
of \cite{dlmf}, 
the relevant  
asymptotic expansions are 
\beq\label{asyAi} 
\Ai\left( \Big(\frac{3\xi}{2}\Big)^{\frac{2}{3}}\right) 
\sim \frac{e^{-\xi}}
{ 2^{\frac{5}{6}} 3^{\frac{1}{6}}\sqrt{\pi}\,\xi^{\frac{1}{6}} }\left(1+\sum_{k=1}^\infty \frac{(-1)^ku_k}{\xi^k}\right),  \qquad 
\Bi\left( \Big(\frac{3\xi}{2}\Big)^{\frac{2}{3}}\right) 
\sim \frac{2^{\frac{1}{6}}e^{\xi}}
{  3^{\frac{1}{6}}\sqrt{\pi}\,\xi^{\frac{1}{6}} }\left(1+\sum_{k=1}^\infty \frac{u_k}{\xi^k}\right), 
\eeq 
each valid as $\xi\to\infty$ in suitable sectors of the 
$\xi$ plane (including the positive real axis), 
for a certain sequence of rational numbers $u_k$. 
In the case at hand, we have 
$$ 
\xi = \frac{1}{4\la}(1+4\la \ze^2)^\frac{3}{2} 
= \frac{1}{4\la}+\frac{3}{2}\ze^2 + O(\la) \quad 
\text{as} \,\, \la\to 0, 
$$ 
so to leading order we 
have 
$$ 
\Ai\left( \Big(\frac{3\xi}{2}\Big)^{\frac{2}{3}}\right) 
\sim c\,e^{-\frac{1}{4\la}-\frac{3}{2}\ze^2}\la^{\frac{1}{6}} \big(1+O(\la)\big), \quad 
\Bi\left( \Big(\frac{3\xi}{2}\Big)^{\frac{2}{3}}\right) 
\sim \tilde{c}\, e^{\frac{1}{4\la}+\frac{3}{2}\ze^2}\la^{\frac{1}{6}} \big(1+O(\la)\big), 
$$
for certain constants $c,\tilde{c}$.
Hence, if we choose $\tilde{A}(\la)$ to cancel out the leading order factors  
$e^{-\frac{1}{4\la}}\la^{\frac{1}{6}} $ in $\Ai$, as well as fixing the 
appropriate normalization, and set $\tilde{B}(\la)=0$ in (\ref{FirstLaxSolution}), then we find a solution $\Psi$ of 
the Schr\"odinger equation 
 given by (\ref{PositivePsiGen}), with leading order 
behaviour 
$$ 
\Psi = 3^{\frac{1}{4}}\zeta^{\frac{1}{2}}
e^{-\frac{3}{2}\ze^2} +o(1) \sim \psi_1 \quad \text{as}\,\,\la\to 0. 
$$
Moreover, because this $\Psi$ 
satisfies (\ref{gensch}), the coefficients of its expansion (\ref{posgenfn}) in powers of $\la$ generate 
precisely the positive Jordan chain (\ref{JordanChainpsi}) with $\delta =1$ (up to the choice of normalization 
constant $a$, which can be adjusted by rescaling $\la$). 
Furthermore, from the dependence of $\Psi$ on $\ze$ and the 
given asymptotic expansion of $\Ai$ at infinity, it is clear that the coefficient $\psi_n$ of  each power of $\la$ is a polynomial in $\ze^2$ multiplied by the 
prefactor $\zeta^{\frac{1}{2}}
e^{-\frac{3}{2}\ze^2}$, as required.
Similarly, by choosing $\tilde{A}(\la)=0$ and fixing 
$\tilde{B}(\la)$ appropriately, we obtain 
the solution (\ref{NegativePsiGen}) with 
$$ 
\tilde{\Psi} = 3^{\frac{1}{4}}\zeta^{\frac{1}{2}}
e^{\frac{3}{2}\ze^2} +o(1) \sim \psi_{-1} \quad \text{as}\,\,\la\to 0,  
$$
which generates the generalized eigenfunctions in the   negative Jordan chain (\ref{JordanChainpsi}) with $\delta =-1$. 
\end{proof}

Observe that, compared with (\ref{Psigen}), the factors 
$e^{\pm\frac{1}{4\lambda}}$ 
appearing in the generating functions $\Psi,\tilde{\Psi}$ above mean that these cannot also satisfy the second ($\la$ flow) part of the Lax pair for Painlev\'e III ($D_7$).  
Also, it is worth mentioning that 
the coefficients $\psi_i$ 
obtained by expanding these generating functions do not all correspond precisely to the generalized eigenfunctions determined in Proposition \ref{psiproposition}, due 
to the freedom to add on multiplies of lower terms in the Jordan chain at each step (which makes no difference to 
the sequence of their Wronskians). It is instructive to 
see how this works in the case of the negative part of  Jordan chain  (\ref{JordanChainpsi}) with $\delta =-1$. 
The 
expansion of $\Bi$ in (\ref{asyAi}) can be rewritten in a more precise form with gamma functions, taking the argument $t\to\infty$ with $|\arg(t)|<\frac{\pi}{3}$, 
as follows:
$$ 
    \Bi(t) \sim \frac{e^{\frac{2}{3}t^{\frac{3}{2}}}}{\sqrt{\pi}t^{\frac{1}{4}}} \sum^{\infty}_{n=0} \left[\frac{\Gamma(n+\frac{5}{6})\Gamma(n+\frac{1}{6})}{2\pi n!t^{\frac{3}{2}n}}\left( \frac{3}{4}\right)^n \right].
$$ 
Using this expression allows the first few terms in 
the expansion of $\tilde{\Psi}$ as $\la\to 0$ to be found explicitly as 
\begin{equation}
    \tilde{\Psi} \sim 3^{\frac{1}{4}}\zeta^{\frac{1}{2}}e^{\frac{3}{2}\zeta^2}\left[1 + \left(\frac{3}{2}\zeta^4-\zeta^2+\frac{5}{18} \right)\lambda + \left(\frac{9}{8}\zeta^8-\frac{5}{2}\zeta^6+\frac{35}{12}\zeta^4-\frac{35}{18}\zeta^2 + \frac{385}{648} \right)\lambda^2 + O(\lambda^3) \right].
\end{equation}
The leading order term (coefficient of $\la^0$)  has been fixed as the normalized eigenfunction 
$\psi_{-1}=3^{\frac{1}{4}}\zeta^{\frac{1}{2}}
e^{\frac{3}{2}\ze^2}$, but reading off the next term, the   coefficient of 
$\la$ gives the generalized eigenfunction 
$$ 
\psi_{-2} =3^{\frac{1}{4}}\zeta^{\frac{1}{2}}
e^{\frac{3}{2}\ze^2} 
\left(\frac{3}{2}\zeta^4-\zeta^2+\frac{5}{18} \right), 
$$
which differs from the formula for $\psi_{-2}$ in (\ref{psinsol}) 
in two ways: first, by virtue of the fact that there the lowest coefficient   $A_{1,0}=0$, whereas above the final non-zero coefficient $\frac{5}{18}$ is included, 
corresponding to the freedom to add a multiple of
$\psi_{-1}$ to $\psi_{-2}$; and second, by an overall factor of $-\sqrt{3}$, since Theorem \ref{OhyamaWronskian} has been stated with the convention that $a=\sqrt{3}$ in \eqref{JordanChainpsi}, 
whereas the form of the equation \eqref{gensch} for $\Psi$ 
corresponds to the choice $a=-1$.  
Similar considerations 
apply to the form of $\psi_{-3}$, given by the 
coefficient of $\la^2$ in the above expansion 
of $\tilde{\Psi}$, which can be simplified by subtracting  linear combinations of $\psi_{-1}$ and $\psi_{-2}$. 

\begin{remark}
It is interesting to note  that the Airy function also appears in two different ways in the theory of Painlev\'e II, as observed in \cite{joshi2004generating}. 
Firstly, (\ref{pii}) has special classical solutions given in terms of Airy functions, when the parameter $\ell$ is an integer (so $\al$ is a half-integer). Secondly, the  rational solutions have a different 
determinantal representation from the one in 
Theorem \ref{piiratthm}, in terms of Hankel determinants, 
and in that context the Airy function $\Ai$ arises in the generating function of the entries of the Hankel matrix.  
We should also like to point out that yet another 
alternative representation for the tau functions of these rational solutions was found recently, in terms of Gram determinants \cite{ling}. Very recently, a broad class of Airy function solutions of KP was constructed, using Grammians \cite{ohta}.  
\end{remark}

\section{Conclusions} 

We have showed how the application of 
confluent Darboux transformations leads to Wronskian representations for algebraic solutions of both the Painlev\'e II equation and the Painlev\'e III ($D_7$) equation. It should be apparent from our presentation 
that there is nothing inherently special about the algebraic solutions, in the sense that every solution of  these equations is connected with  a Jordan chain of generalized eigenfunctions for a sequence of Schr\"odinger operators, linked to one another by the repeated action of a confluent Darboux  transformation. Hence every sequence of solutions of these Painlev\'e equations, related to one another by iterated BTs, admits a formal Wronskian representation. The reason we use the word ``formal'' here is that the generalized eigenfunctions, which appear as entries in the Wronksians, are typically higher transcendental functions, constructed from a Schr\"odinger operator whose potential is built out of Painlev\'e transcendents. So this is what distinguishes the algebraic solutions from the general case: for algebraic solutions,  the entries of the Wronskians can be expressed in closed from, and/or as polynomials, 
and similar considerations apply to classical solutions of Painlev\'e equations, such as the Airy 
solutions of Painlev\'e II, which are expressed in terms of 
Wronskians of Airy functions and their derivatives \cite{okamoto}; but for the general 
transcendental solution, the seed eigenfunctions appear to be at least as complicated as the solution itself.

In recent work, which  will appear shortly, we present two more (different) determinantal representations for the algebraic solutions of the Painlev\'e III ($D_7$) equation. One of these new  representations 
is a Hankel determinant formula, 
generalizing known results on Hankel determinants for solutions of Painlev\'e II \cite{Kajiwara_1999}, 
that arises due to the connection of 
Painlev\'e III ($D_7$) with the Toda lattice, 
which we are able to exploit based on 
the bilinear equation  (\ref{tauToda}), 
using standard forms for  Toda solutions 
(see e.g. \cite{KajiwaraToda, kajiwara2007remark}).  
The other new determinantal representation for the 
Ohyama polynomials is an analogue of a determinant 
of Jacobi-Trudi type in terms of generalized Laguerre polynomials, 
which was presented to us in the form of a longstanding  conjecture 
by Kajiwara \cite{private}. A particular advantage of the latter 
representation is that the properties of the Ohyama polynomials in terms of the variable $s$ are made manifest. 


\section*{Appendix: Intertwining relations and Wronskians} 
%
\label{IntertwiningSection}

\renewcommand{\theequation}{A.\arabic{equation}}


The purpose of this appendix is to collect together some algebraic facts about Jordan chains for confluent Darboux 
transformations. In particular, we 
focus on the specific role played by the 
Burchnall-Chaundy relation, and how it relates to the definition of the Jordan chain and the choice of  normalization for the generalized eigenfunctions. 
This is especially pertinent to Adler and Moser's construction of the rational solutions of KdV, as in section 3, and to our derivation of 
the Wronskians for the algebraic solutions of 
Painlev\'e III ($D_7$). 
While most of the facts presented here  
can probably be found somewhere in the extensive literature on Darboux transformations, 
it appears that the 
precise connection between the normalization 
of the Jordan chain and the Burchnall-Chaundy relation 
was not addressed in the recent works \cite{contreras2017recursive, schulze2013wronskian} on the confluent case. 

Suppose that we have a sequence of Schr\"odinger operators, labelled by $n\in\Z_{\geq 0}$, related to one 
another by the repeated application of confluent Darboux transformations to some initial operator. In order  
to eliminate some inconvenient minus signs, we use $H_n$ to denote the 
(formally self-adjoint) operator 
$$ 
H_n = -\left(\frac{\dd^2}{\dd z^2}+V_n\right).  
$$
So we 
can regard each $H_n$ as the 
Hamiltonian operator for some 1D quantum mechanical system (with $-V_n$ being the potential energy, in that context). 

In order to construct the sequence of operators via the action of confluent Darboux transformations, we consider an eigenfunction $\phi_0$ for the initial operator $H_0$, that is  
$$ 
H_0\,\phi_0 = 0, 
$$ 
with eigenvalue zero. (Here 
we consider only the case of repeated eigenvalue $\mu=0$, because that is the only case of interest in
the rest of the paper, but the general case requires only minor modifications; cf.\ section 2.) 
Then we have the factorization 
$$ 
H_0 = L_0^\dagger L_0, \qquad 
L_0=\frac{\dd }{\dd z} -\frac{\phi_0'}{\phi_0}, 
\quad  L_0^\dagger=-\frac{\dd }{\dd z} -\frac{\phi_0'}{\phi_0},
$$
so that $\phi_0$ spans the kernel of the first order operator $L_0$, i.e. 
$$L_0\, \phi_0=0, $$
and the next operator is produced by reversing the order of factorization: 
$$
H_1 = L_0 L_0^\dagger. 
$$ 
So far we just have a standard (single) Darboux transformation, and by construction the reciprocal of 
the original eigenfunction is an eigenfunction of the new operator, that is 
$$ 
H_1 \, (\phi_0)^{-1} =0. 
$$ 
Then the main question is how to iterate the construction and obtain a sequence of eigenfunctions $\phi_n$ for each operator, with the same eigenvalue (in this case, zero), 
when the standard formula \eqref{basetransform} for 
generating a new eigenfunction breaks down, and the usual Crum formula is no longer valid.

The answer is to build $\phi_n$ from a sequence of generalized eigenfunctions for $H_0$, denoted $\psi_n$, which arise in the  following way. Starting with the 
original eigenfunction of $H_0$, that is  
$$\phi_0 = \psi_1,$$ 
we define the new eigenfunction for $H_1$ to be 
$$ \phi_1=L_0 \,\psi_2,  
$$
so we find 
$$ H_1\, \phi_1 = L_0L_0^\dagger\,\phi_1 = L_0 L_0^\dagger L_0\, \psi_2=L_0 H_0\,\psi_2  \implies H_0\,\psi_2 \in \mathrm{ker} L_0.  
$$
Hence $H_0\psi_2$ is a multiple of $\phi_0=\psi_1$, which gives the generalized eigenfunction equation 
$$ 
\left(\frac{\dd^2}{\dd z^2}+V_0\right)^2\psi_2 =
H_0^2 \,\psi_2 =0. 
$$

However, we can say more: for a non-trivial result at the next stage of iteration, we require that $\phi_1$ should be independent of $\phi_0^{-1}$, the other eigenfunction of $H_1$ that we 
already know; 
thus  their Wronskian is non-zero: 
\beq\label{wr1}
Wr\big( \phi_1, (\phi_0)^{-1}\big) 
= C_1\neq 0.
\eeq 
However, we can rewrite the latter condition in terms of 
operators, as 
$$ 
(\phi_0)^{-1}\, L_0^\dagger \, \phi_1 = C_1 
\implies  L_0^\dagger \, \phi_1 = C_1 \, \phi_0, $$
and replacing $\phi_1$ and $\phi_0$ in terms of $\psi_2$ and $\psi_1$, respectively, this becomes 
$$  L_0^\dagger L_0 \, \psi_2 =C_1\, \psi_1 
\iff H_0 \, \psi_2  =C_1\, \psi_1 , 
$$
which is precisely the Jordan chain condition 
\eqref{Jchain}, but with an arbitrary choice of non-zero 
normalizing constant $C_1$, 
coming from the Wronskian \eqref{wr1}; the formulae in sections 2 and 3 
involved the particular 
choice $C_1=-1$, 
whereas in section 4, in the context of the Ohyama polynomials, the choice $C_1=\sqrt{3}$ was made.

For subsequent eigenfunctions, we proceed by induction. For each $n$, we have a Schr\"odinger operator factorized in two different ways, that is 
\beq \label{Hfac}
H_n  = L_{n-1}L_{n-1}^\dagger= L_n^\dagger L_n, 
\eeq 
which gives the standard intertwining relation 
\beq\label{twine} 
L_{n-1}H_{n-1} = H_n L_{n-1}, 
\eeq 
and at each iteration of the confluent 
Darboux transformation, the new eigenfunction is defined by 
\beq\label{phindefn} 
\phi_n = M_n \, \psi_{n+1}, \qquad \text{where} \quad M_n =L_{n-1}L_{n-2}\cdots L_0. 
\eeq 
By construction, the reciprocal of the 
eigenfunction from the previous step satisfies 
$$ 
L_{n-1}^\dagger\, (\phi_{n-1})^{-1}=0 \implies 
H_n\, (\phi_{n-1})^{-1} = 0 
$$
by (\ref{Hfac}), and we require that the new eigenfunction should be independent, with the Wronskian   
\beq\label{wrn}
Wr\big( \phi_n, (\phi_{n-1})^{-1}\big) = C_n, 
\eeq 
for arbitrary $C_n\neq 0$. Similarly to (\ref{wr1}), 
the latter condition can be rewritten as the operator equation 
\beq\label{nextldag} 
L_{n-1}^\dagger \, \phi_n = C_n\, \phi_{n-1},  
\eeq 
and, after  applying $L_{n-1}$ to both sides, this implies 
$$
L_{n-1}L_{n-1}^\dagger \, \phi_n =  L_{n-1} (  C_n\,\phi_{n-1}) \implies 
 H_n\, \phi_n = 0 . 
$$
Furthermore, by \eqref{phindefn} we can rewrite \eqref{nextldag} as   
$$ 
L_{n-1}^\dagger M_n\, \psi_{n+1} = C_n \, M_{n-1}\, \psi_n, 
$$
and then note that, by repeated application 
of the intertwining relation (\ref{twine}), we have
$$
\begin{array}{rcl}
   L_{n-1}^\dagger M_n & = & L_{n-1}^\dagger L_{n-1} L_{n-2}\cdots L_{0} \\ 
   & =& H_{n-1}M_{n-1}\\ 
   & =& L_{n-2}H_{n-2}M_{n-2}\\ 
   & =& L_{n-2}L_{n-3}H_{n-3}M_{n-3}\\ 
   & = &\ldots \\
   & = & M_{n-1}H_0, 
\end{array}
$$
which means that another consequence of \eqref{nextldag} 
is the equation 
$$
M_{n-1}H_0\, \psi_{n+1} = C_n\, M_{n-1}\, \psi_n, 
$$
or equivalently 
\beq\label{kerm}
H_0\, \psi_{n+1} - C_n\, \psi_n \in \ker M_{n-1}. 
\eeq 
Also note that, at the next step of the Darboux transformation, we have 
\beq\label{Mker} 
L_n\, \phi_n =0 \implies L_n M_{n-1}\, \psi_{n+1} 
= M_{n}\, \psi_{n+1} =0. 
\eeq

The statement (\ref{kerm}), for each $n$, is the most general formulation 
of the Jordan chain condition, which for $i=n+1$ includes \eqref{Jchain} as a special case. To understand the general condition, observe that 
$$M_n=\left(\frac{\dd}{\dd z} -\frac{\phi_{n-1}'}{\phi_{n-1}}\right) \cdots 
\left(\frac{\dd}{\dd z} -\frac{\phi_{0}'}{\phi_{0}}\right) 
=\frac{\dd^n}{\dd z^n}+\cdots 
$$ 
is a differential 
operator of order $n$, and we claim that its action on 
any function $\chi$ is given as a ratio of Wronskians, by 
the formula 
\beq\label{mwro}
M_{n}\, \chi =\frac{Wr(\psi_1,\ldots,\psi_n,\chi)}
{Wr(\psi_1,\ldots,\psi_n)}. 
\eeq 
To see why the above formula is correct, observe that from
$M_{n} = L_{n-1}M_{n-1}$ and \eqref{Mker} it is clear by induction that $\mathrm{ker}M_{n}$ is spanned by $\psi_1,\ldots,\psi_n$, and the Wronskian in the numerator vanishes whenever $\chi$ is a linear combination 
of these functions; moreover, the linear operator 
defined by (\ref{mwro}) also has
the same 
leading term as $M_n$, so they must be the same.  It follows from this that the most general 
way to satisfy the condition (\ref{kerm}) is to take, at each step,  
\beq\label{jocha} 
H_0\, \psi_{n+1} = C_n \, \psi_n + \sum_{j=1}^{n-1} 
c_{n,j}\, \psi_j, 
\eeq 
with additional arbitrary constants $c_{n,j}$.  
However, while the freedom to add these extra terms is always present, it can be absorbed into a redefinition 
of $\psi_n$ on the right-hand side of (\ref{jocha}), 
which does not affect the values of the Wronskians.

Finally, by combining (\ref{phindefn}) with (\ref{mwro}),  
it follows that the eigenfunctions of each 
Schr\"odinger operator $H_n$ are given by the ratio of Wronskians \eqref{Jordanchi}, that is 
$$
\phi_n=\frac{\theta_{n+1}}{\theta_n}, 
\qquad \mathrm{with} 
\quad \theta_n = Wr(\psi_1,\ldots,\psi_n), 
$$
and the Wronskian condition \eqref{wrn} for the pair of 
eigenfunctions obtained at each step is 
equivalent to the Burchnall-Chaundy relation in the form 
$$ 
\theta_{n-1}'\theta_{n+1}-\theta_{n-1}\theta_{n+1}' 
=C_n\, \theta_{n}^2,  
$$
where the normalization constant $C_n\neq 0$ can be freely chosen for each $n$.


\end{document}